\documentclass{article}

\usepackage[a4paper, margin=2.0cm]{geometry}
\usepackage{amsmath,amssymb,amsthm}
\newtheorem{theorem}{Theorem}
\newtheorem{corollary}{Corollary}
\newtheorem{lemma}{Lemma}
\newtheorem{proposition}{Proposition}
\newtheorem{remark}{Remark}

\DeclareMathOperator*{\argmax}{argmax}

\usepackage{pdfpages}

\usepackage{cancel}

\usepackage{marginnote}
\usepackage{bm}

\newcommand{\mA}{{\mathcal A}}
\newcommand{\mB}{{\mathcal B}}
\newcommand{\mC}{{\mathcal C}}

\newcommand{\mJ}{{\mathcal J}}
\newcommand{\mK}{{\mathcal K}}

\newcommand{\mQ}{{\mathcal Q}}
\newcommand{\procQ}{\bm{{\mathcal Q}}}

\newcommand{\mS}{{\mathcal S}}

\newcommand{\mCl}{{\mathcal Cl}}

\newcommand{\bN}{{\mathbb N}}
\newcommand{\bR}{{\mathbb R}}
\newcommand{\bZ}{{\mathbb Z}}

\newcommand{\bE}{{\mathbb E}}

\newcommand{\vecQ}{{\bf Q}}

\newcommand{\vecpi}{{\boldsymbol \pi}}

\newcommand{\vecrho}{{\boldsymbol \rho}}

\newcommand{\veclambda}{{\boldsymbol \lambda}}

\newcommand{\veca}{{\boldsymbol a}}

\newcommand{\vecsigma}{{\boldsymbol \sigma}}

\usepackage{authblk}
\author[1]{M. Bramson}
\author[2]{B. D'Auria}
\author[3]{N. Walton}

\affil[1]{School of Mathematics, University of Minnesota. {bramson@math.umn.edu}}
\affil[2]{Departmento de Estad\'istica, Universidad Carlos III de Madrid. {bernardo.dauria@uc3m.es}}
\affil[3]{Alan Turing Building, University of Manchester. {neil.walton@manchester.ac.uk}}

\usepackage{enumitem}
\usepackage{soul} 
\usepackage{tikz}
\usepackage[normalem]{ulem}
\usetikzlibrary{shapes.misc}
\usepackage{todonotes} 
\usepackage{xfrac}

\usepackage[colorlinks=true,breaklinks=true,bookmarks=true,urlcolor=blue, 
     citecolor=blue,linkcolor=blue,bookmarksopen=false,draft=false]{hyperref}

\def\EMAIL#1{\href{mailto:#1}{#1}}

\begin{document}







\title{Stability and Instability of the MaxWeight Policy}

\maketitle
\begin{abstract}
	Consider a switched queueing network with general routing among its queues.
	The MaxWeight policy assigns available service by maximizing the objective function $\sum_j Q_j \sigma_j$ among  the  different  feasible  service  options,  where $Q_j$ denotes  queue  size  and $\sigma_j$ denotes  the  amount  of  service  to  be  executed  at   queue $j$.
	MaxWeight  is  a  greedy  policy that does not depend on knowledge of arrival rates and is straightforward to implement.
	These properties,  as  well  as  its  simple  formulation,  suggest  MaxWeight  as  a  serious  candidate  for implementation in the setting of switched queueing networks; it has been extensively studied in the context of communication networks.
	However, a fluid model variant of MaxWeight was shown by Andrews--Zhang (2003) not to be maximally stable.
	Here, we prove that MaxWeight itself is not in general maximally stable.
	We also prove MaxWeight is maximally stable in a much more restrictive setting, and that a weighted version of MaxWeight, where the weighting depends on the traffic intensity, is always stable. 	
\end{abstract}


%



\section{Introduction.}\label{sec:intro}






The MaxWeight policy has been extensively employed for the past two decades in the setting of switched networks, in particular, for communication networks.  It is a discrete time policy that allocates service of jobs at each time step by maximizing a corresponding objective function over a given
set $\mS$ of feasible service options, or \textit{schedules}.

Let $\mJ$ denote the set of queues of the network, which are assumed to be single class; let $Q_j$, $j\in \mJ$, be the corresponding queue lengths at a given time; and let $\sigma_j$ be the nonnegative integer valued
amount of service to be executed at the queue at this time.  Then the \textit{MaxWeight policy} chooses the schedule that solves the optimization problem

\begin{align}
\label{eq1}
&\text{maximize}\quad\sum_{j\in\mJ} Q_j {\sigma_j}\quad
\text{over} \quad  {\vecsigma} \in {\mS}\,.
\end{align}

In the case of ties in (\ref{eq1}), $\vecsigma$ may be chosen arbitrarily (or randomly) among these schedules.
We denote  by $\vecpi(\vecQ)$ the schedule of the MaxWeight policy for the queue length vector
$\vecQ := (Q_j:j\in\mJ)$, and denote by ${\procQ}=(\vecQ(t):\, t\in \bZ_+)$
the associated \textit{queueing network process} defined by the MaxWeight policy. 
The queueing network process is a discrete time countable state space Markov chain satisfying the Strong Markov Property. 
{For the set $\mS \subset \bZ^{|\mJ|}_+$,} we will always assume that
$\vecsigma \in \mS$ implies that $\vecsigma' \in \mS$, for
$\sigma ' \in \bZ^{|\mJ|}_+$ with
 $\vecsigma'  \le \vecsigma$ (i.e., $\sigma'_j \le \sigma_j$ for $j\in \mJ$).
 Here, as in the remainder of the paper, $\bZ_+ := \{0,1,2,\ldots\}$, whereas  $\bN := \{1,2,\ldots\}$.

The MaxWeight policy has the virtues of being simple to formulate and of its implementation not depending on knowledge of the arrival rates of jobs into the network.  We also note that the Longest-Queue-First-Served (LQFS) policy is a special case of MaxWeight. 

Rigorous results in the MaxWeight literature typically are for single-hop networks, that is, networks where each job leaves the network immediately after its service at any queue. The seminal paper of Tassiulas and Ephremedes \cite{TaEp92} showed that MaxWeight is \textit{maximally stable} for single-hop networks.
That is, the associated queueing network process {{$\procQ$}} is positive recurrent whenever the traffic intensity vector $\bm{\rho}$ is contained
in the interior of the convex hull of $\mS$.

 {MaxWeight was studied for single-hop networks in Tassiulas and Ephremedes \cite{tassiulas1993dynamic} in the context of wireless networks and in McKeown et al. \cite{MMAW99} in the context of internet router design.
 Since these works, MaxWeight and its generalizations have been extensively studied in the context of single-hop switched networks. We refer the reader to the following surveys and books \cite{georgiadis2006resource,jiang2010scheduling,kelly2014stochastic,neely2010stochastic,srikant2013communication}, which provide a broad review of the analysis, extensions and applications of MaxWeight.}

In the multihop setting, that is, {where jobs completing service at a queue can be routed to another queue, there is little literature for the MaxWeight policy. {However,} there is considerable literature for the somewhat related BackPressure policy.  As its name suggests,  the BackPressure policy chooses the schedule $\vecsigma \in \mS$ where the sum of the difference in queue lengths between consecutive queues is maximized; the policy is designed to balance the lengths among different queues. It was shown in \cite{TaEp92} that BackPressure is maximally stable.  This result in fact includes the multiclass setting as well as the single class setting, when the BackPressure algorithm is applied to individual classes, rather than to entire queues.  Other literature on BackPressure includes \cite{georgiadis2006resource} and \cite{ShWi11}.
We note that BackPressure requires knowledge of the routing probabilities between queues.

In contrast to the BackPressure policy, rigorous stability results are lacking for MaxWeight in the multihop setting.
 Authors of the current paper were approached on more than one occasion about this stability, {of which we were at the time ignorant.}
%

However,
in a relatively unknown paper, Andrews and Zhang \cite{andrews2003achieving} studied the fluid model analog of a variant of MaxWeight for a switched network possessing {state-dependent} arrivals, and showed  that this fluid model is not maximally stable.  (The modulated arrival rate is employed to produce instability.)  The paper also provided simulations indicating similar unstable behavior for {stationary} arrivals in {its} switched network analog.

In the current paper, we will prove that the MaxWeight policy itself is not maximally stable for {multihop} single class switched networks. 
This is the first mathematical proof of transience of MaxWeight under subcritical arrival rates.
A counterexample is given by Theorem \ref{thrm:MW} later in the introduction; the detailed argument will be given in Section  \ref{sec:proofs}, after certain definitions are presented in Section \ref{sec:model}. 
Since Section \ref{sec:proofs} is quite involved, we recommend that the reader skip it on a first reading.   The argument follows in spirit the constructions in Lu--Kumar \cite{LuKu91} and Rybko--Stolyar \cite{RySt92}.

Positive recurrence of subcritical multihop networks with
 the Longest-Queue-First-Served policy is demonstrated in the literature in specific situations (see \cite{Pedarsani2016a,Dimakis2006a,Baharian2011a}, and the included references).
In Appendix \ref{appendixD}, we reinterpret the LQFS policy as a special case of MaxWeight, and present an example  that is transient; the argument is analogous to that of the above counterexample for MaxWeight. 

%
	%
  
Although it seems that general stability results for MaxWeight are the exception rather than the rule for switched networks,
we will show that MaxWeight is maximally stable in the restrictive setting of a tandem single class switched network with equal mean service times at different queues.  This result follows from the more general result that, for any single class switched network,
an appropriately weighted version of MaxWeight will always be stable, with the weighting at each queue
depending on the traffic intensity at the queue. This weighting reflects the underlying structure of the MaxWeight policy and provides insight as to why unweighted MaxWeight is itself not the mathematically correct structure to ensure positive recurrence of the associated queueing network process.   These results are summarized
in Theorems \ref{thrm:p-MW} and \ref{tandemtheorem} later in the introduction, with more detail being given in Section \ref{sec:stable}.
These results are proved by using fluid models and constructing appropriate Lyapunov functions for these fluid models.  The required fluid model machinery is summarized in the Appendix.

So far, we have considered only the stability of switched queueing networks that are single class.  In Section \ref{sec:multiclass}, we briefly discuss
 the stability of multiclass switched networks that are FIFO within each queue.  Conditions for stability of multiclass switched networks are more elusive than those in the single class setting. We simulate a queueing network that is unstable, although it is of Kelly type (i.e., the mean service rates at all classes within a queue are equal).  
We do not provide the proof, which is a more lengthy version of that of Theorem \ref{thrm:MW}. 
We then show that a modification of the Lyapunov function from the single class setting implies stability of multiclass switched networks  for a particular variant of MaxWeight. 


In Section \ref{sec:Conc}, we will briefly compare the stability of the MaxWeight policy for multiclass switched networks with that of the \emph{ProportionalScheduler}, where the terms $Q_j \sigma_j$ in (\ref{eq1}) are replaced by $Q_j \log \sigma_j$.  As shown in 
 Bramson et al. \cite{BDW17}, the ProportionalScheduler is maximally stable for all multiclass multihop switched networks that are of Kelly type (and hence is automatically maximally stable for all single class multihop switched networks).  We will give
elementary heuristic reasoning  why the term $\log \sigma_j$ produces a more stable policy than does $\sigma_j$.


In the remainder of the introduction, we present Theorems \ref{thrm:MW}--\ref{tandemtheorem}. 



\subsection*{An example showing the instability of MaxWeight.}
In Section \ref{sec:proofs}, we will prove instability  for multihop switch networks with the MaxWeight policy.  We now state this result, Theorem \ref{thrm:MW}.   We refer the reader to Section \ref{sec:model} for certain precise definitions and conditions that are required for the theorem. 

The traffic intensity $\bm{\rho}=(\rho_j :j\in\mJ)$ of a switched network {{is a fundamental concept in queueing theory and}} gives the long-run average rate at which work destined for each queue arrives in a network; it is defined in Section \ref{sec:model}. In order {for work not to accumulate
somewhere in the network}},
$\bm{\rho}$ must also equal the long-term average service rates {attained by {applying schedules in $\mS$ to the queueing network process  {$\procQ$}}}.
The \textit{subcritical region} $\mC$ is the interior of the convex hull $<\mS>$ of $\mS$, and a fixed policy is defined to be \textit{maximally stable} if the queueing network process {$\procQ$} is positive recurrent whenever
$\vecrho\in \mC$.


In Theorem \ref{thrm:MW}, we employ the single class switched network depicted in Figure \ref{Fig2}.
\begin{figure}
	\centering
%
\begin{tikzpicture}[scale=1]
\def\hsp{0.25} 
\def\vsp{0.5} 
\def\arrL{2} 
\def\stSep{4.75}

\tikzset{
    pic shift/.store in=\shiftcoord,
    pic shift={(0,0)},
    queue/.pic = {
        \begin{scope}[shift={\shiftcoord}]
            \coordinate (_arr) at (-0.25,0);
            \coordinate (_in) at (0.5,0);
            \coordinate (_dep) at (1,0);
            \draw (0,0.25) -- (1,0.25) -- (1, -0.25) -- (0,-0.25);
         \end{scope}
	}
}

	\node (A) [draw, thick, shape=rectangle, minimum width=2cm, minimum height=4cm, anchor=center] 
		at (0,0)  {\raisebox{9em}{$\mathcal{A}$}};
	\path pic [rotate=180] (A0) at (0.5,0.875) {queue} (A0_in) node {$\mathcal{A}_0$};
	\path pic (A1) at (-0.5,-0.25) {queue} (A1_in) node {$\mathcal{A}_1$};
	\path pic (AJ) at (-0.5,-1.50) {queue} (AJ_in) node {$\mathcal{A}_J$};
	\path (0,0) -- ++(0,-0.75) node {$\vdots$};
	\draw [<-] (A1_arr.west) -- ++(-1,0) node [left]{$a/J$};
	\draw [<-] (AJ_arr.west) -- ++(-1,0) node [left]{$a/J$};
	\draw [->] (A0_dep.east)  -- ++(-1.5,0);	

	\node (B) [draw, thick, shape=rectangle, minimum width=2cm, minimum height=4cm, anchor=center] 
		at (3.25,0) {\raisebox{-10em}{$\mathcal{B}$}};
	\path pic [rotate=180] (BJ) at (3.75,1.50) {queue} (BJ_in) node {$\mathcal{B}_J$};
	\path (0,0) -- ++(3.25,1) node {$\vdots$};
	\path pic [rotate=180] (B1) at (3.75,0.25) {queue} (B1_in) node {$\mathcal{B}_1$};
	\path pic (B0) at (2.75,-0.875) {queue} (B0_in) node {$\mathcal{B}_0$};
	\draw [<-] (B1_arr.west) -- ++(1,0) node [right]{$a/J$};
	\draw [<-] (BJ_arr.west) -- ++(1,0) node [right]{$a/J$};
	\draw [->] (B0_dep.east)  -- ++(1.5,0);

	\draw [->] (A1_dep.east)  -- ([shift={(0,+0.25em)}]B0_arr.east);
	\draw [->] (AJ_dep.east)  -- ([shift={(0,-0.25em)}]B0_arr.east);
	\draw [->] (BJ_dep.east)  -- ([shift={(0,+0.25em)}]A0_arr.east);
	\draw [->] (B1_dep.east)  -- ([shift={(0,-0.25em)}]A0_arr.east);

\end{tikzpicture}
%
	\quad
%
\begin{tikzpicture}[scale=0.7]


	\draw[fill=gray!20] (0,0) -- 
		(1,0) node[below]{\tiny $1$} -- 
		(0,3) node[left]{\tiny $\nu$} -- 
		cycle;

	\draw[->, thick] (0,0)--(3.5,0) node[right]{$\sigma_{\mA_0}$};
	\draw[->, thick] (0,0)--(0,3.5) node[above]{$\sum_{j \geq 1} \sigma_{\mA_j}$};
	

		
\end{tikzpicture}
%
	\caption{On the left, the switched network in Theorem \ref{thrm:MW} with parameters $a$ and $J$.
	On the right is the projection of the set $<\mS>$ on the space $(\sigma_{\mA_0},\sum_{j\ge 1}\sigma_{\mA_j})\in\bR^2$.
	\label{Fig2}}
\end{figure}
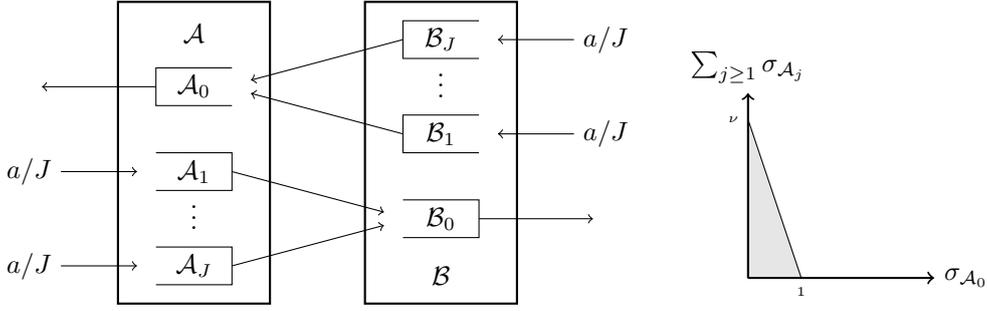
For given $J\in \mathbb{Z}_+$, there are two sets of queues, {$\{ \mA_1,...,\mA_J\}$} and {$\{\mB_1,...,\mB_J\}$}, with arrivals occurring at each end of the network, at rate $a$, that choose randomly between the corresponding $J$ queues. 
We assume that, over each time unit, there is at most $1$ external arrival at each queue, and that arrivals are i.i.d. at different times.
 There are also two queues {$\mA_0$} and {$\mB_0$} that
 receive jobs after their service is completed at the respective queues {{$\{\mB_1,...,\mB_J\}$} and {$\{ \mA_1,...,\mA_J\}$}}.
{For convenience, we denote by $\mA$ and $\mB$ the {\textit{components}} consisting of the unions of all queues of the form {$\mA_j$ and $\mB_j$, $j=0,...,J$, {respectively}}.  

We assume that jobs are unit sized, and so the traffic intensity $\vecrho$ is given by
\begin{align}
\label{displayforrho}
\rho_{\mA_0} = \rho_{\mB_0} =  a,\qquad\text{and}\qquad \rho_{\mA_j} = \rho_{\mB_j} = \frac{a}{J}\, , \quad j=1,...,J\,.
\end{align}
{
For given $\nu \in \bN$, we specify the set of feasible schedules $\mS$ in (\ref{eq1}) by the conditions on $\mA$ and $\mB$,
}

\begin{align}
\label{Sfortheorem1}
{\sigma_{\mathcal{A}_0}}{}
+
\frac{1}{\nu}
\sum_{j=1}^J
	\sigma_{\mA_j}
\leq 1
\qquad  \text{and} \qquad
{\sigma_{\mathcal{B}_0}}
+
\frac{1}{\nu}\sum_{j=1}^J
	\sigma_{\mB_j}
\leq 1\, .
\end{align}
In other words, during each time unit,
at each nonempty component either up to $\nu$ jobs are served from queues containing jobs that arrive from outside the network or a single job is served from the queue serving jobs that arrive internally.  
It follows from displays (\ref{displayforrho}) and (\ref{Sfortheorem1}) that $\vecrho \in \mathcal C$ whenever
\begin{equation}\label{superqueueload}
r_{\vecrho} := a \left( 1 + \frac{1}{\nu} \right) < 1 \, .
\end{equation}

Theorem \ref{thrm:MW} provides conditions  on the parameters $a$, $\nu$, and $J$ such that the queueing network processes for the switched networks in Figure \ref{Fig2} are subcritical but  transient under the MaxWeight policy.

\begin{theorem}
\label{thrm:MW}
 Consider the MaxWeight policy for the single class {multihop} switched network represented by Figure \ref{Fig2} and with $\mS$ given by (\ref{Sfortheorem1}).   Assume that (\ref{superqueueload}) is satisfied and that 
\begin{align}
\label{thm1condition}
1 < \nu < J 
\qquad \text{and} \qquad
\frac{J}{2J-\nu} < a < 1 - \frac{(J+\nu)(J+\nu^2)}{\nu(J^2+ J+ \nu^2)} \ .
\end{align}
Then $\vecrho\in \mC$ and the associated queueing network process {$\procQ$} is transient.
\end{theorem}

It is easy to choose $a$, $\nu$, and $J$ so that both (\ref{superqueueload})
and~\eqref{thm1condition} are satisfied.
For instance, one can choose
$a=7/12$, $\nu=6$, and $J=30$. (A simulation under these parameters is provided in Figure \ref{Fig:Sim1}.)
It follows that the MaxWeight policy is not in general maximally stable for single class {multihop} switched networks.
This result is proved in Section \ref{proof}.



\begin{figure}
	\centering
	\includegraphics[width=0.95\textwidth]{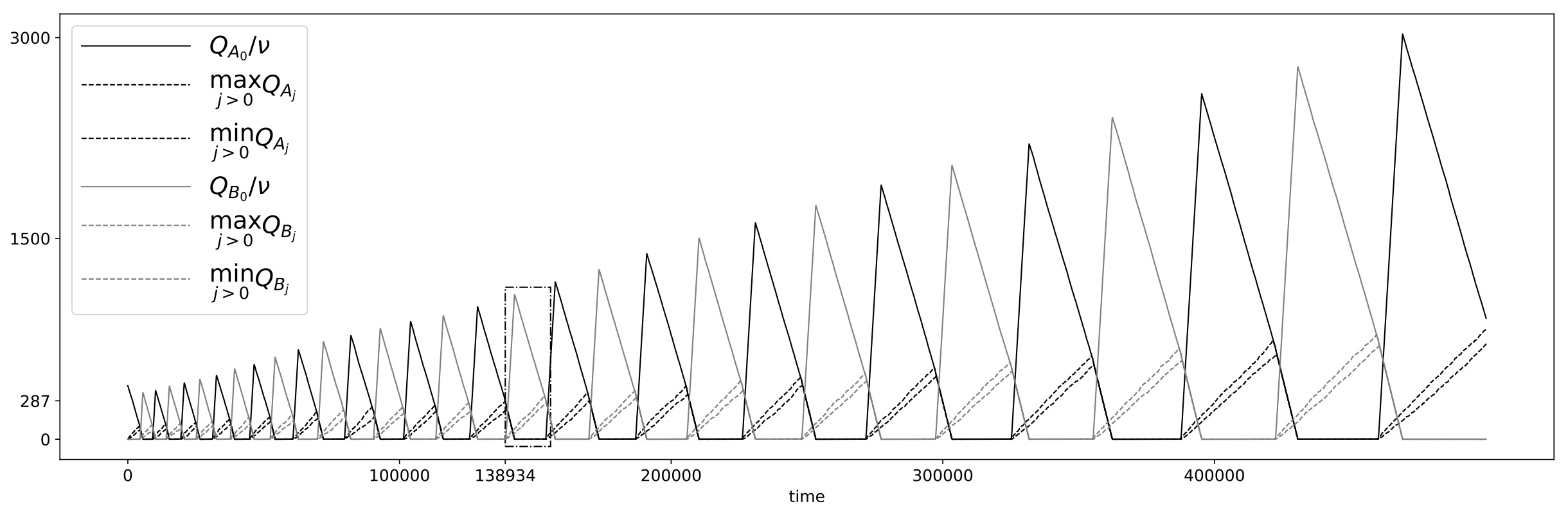}
	\caption{A simulation of the network described in Theorem \ref{thrm:MW}, for parameters
	$a=7/12$, $\nu=6$, and $J=30$. The initial value is $Q_{\mA_0}(0)=1722$, with all other queues empty.
	The rectangular box in the figure outlines a region that is presented in detail in Figure \ref{Fig:Sim1Cut}.
	\label{Fig:Sim1}}
\end{figure}

\subsection*{Conditions under which MaxWeight and weighted MaxWeight are stable.}
The example in Theorem \ref{thrm:MW} shows that  MaxWeight is not in general maximally stable for single class multihop switched networks.  Nevertheless,
it is maximally stable in certain more restrictive settings, such as for tandem single class switched networks.  This result is a consequence of Theorem \ref{thrm:p-MW}, which is a stability result for networks with general routing. Here, the MaxWeight policy is replaced by a \textit{weighted MaxWeight policy} whose weighting at each queue is given by the traffic intensity at that queue.

For this weighted MaxWeight policy, the optimization problem in (\ref{eq1}) is replaced by the optimization problem
\begin{align}
\label{eq1wtmaxwt}
&\text{maximize}\quad\sum_{j\in\mJ} Q_j \frac{\sigma_j}{\rho_j} \quad
\text{over} \quad  {\vecsigma} \in {\mS}\,.
\end{align}
The network is assumed to be an open network, with arbitrary \textit{mean routing matrix} ${\bf P} = (P_{ij}: i,j \in \mJ)$.  (More detail will be given in Section 2.)

%
\begin{theorem}
\label{thrm:p-MW}
Consider the weighted MaxWeight policy given by (\ref{eq1wtmaxwt}), for a single class switched network with an arbitrary set of schedules $\mS$ and arbitrary mean routing matrix $\bf P$.  If $\bm {\rho} \in \mC$, then the associated queueing network process $\bm{\mathcal{Q}}$ is positive recurrent.
\end{theorem}

The dependence of the objective function in (\ref{eq1wtmaxwt}) on $\bm{\rho}$ restricts the applicability of this weighted version of MaxWeight unless $\bm{\rho}$ is first sampled or known.  
For this reason, weighted MaxWeight is not maximally stable in a true sense. Nonetheless, the weighting in (\ref{eq1wtmaxwt}) corresponds naturally to the underlying structure of MaxWeight and so indicates why unweighted MaxWeight is not in general the correct condition to ensure maximal stability.

Suppose now that the network is a \textit{tandem network}, {i.e., all jobs enter the network at a single queue, where $P_{j,j+1} =1$ for all $1\le j < |\mJ|$, and all jobs leave the network after completing service at the queue $|\mJ|$.}
Also, assume that the mean amount of work $m$ required to successfully serve a job at a queue is the same for all queues.  Then the traffic intensity $\rho_j$ is constant over all queues (irrespective of the choice of $\mS$) and so the optimization problem in (\ref{eq1wtmaxwt}) reduces to the unweighted MaxWeight optimization problem in (\ref{eq1}).  A direct application of Theorem \ref{thrm:p-MW} therefore implies the following result.
\begin{theorem}
\label{tandemtheorem}
Consider a tandem switched network with an arbitrary set of schedules $\mS$ and the same mean work per job required at each queue.  Then the MaxWeight policy given by (\ref{eq1}) is maximally stable for this switched network.
\end{theorem}

Since Theorem \ref{tandemtheorem} is easier to understand than Theorem \ref{thrm:p-MW}, we will first prove it directly and then prove Theorem \ref{thrm:p-MW}.  Both results are proved in Section \ref{sec:stable}.

Theorem \ref{thrm:p-MW} can be applied to the MaxWeight policy for general network topologies when the traffic intensity at all queues is the same. The MaxWeight policy is 
maximally stable in certain other settings as well.  An elementary case is when the network consists of parallel sequences of queues in tandem that do not interact, if one assumes that the mean required work is constant over queues along individual sequences.  
 Switched networks whose topology allows branching of routes (but no merging) are  more significant examples and are considered at the end of Section \ref{sec:stable}.

We point out that, although tandem networks are not reasonable models for communication networks, they are often applied in operations research to analyze production systems, such as silicon chip manufacturing. Also, routes with branching occur in distribution networks.

\section{Additional Notation and Model Description.}\label{sec:model}
Here, we introduce further notation needed to describe single class switched queueing networks.
We specify {the} queueing network equations, {define the} arrival rates $\bm{\lambda}$ and traffic intensities $\bm{\rho}$, and briefly discuss the subcritical region $\mathcal C$.
(Multiclass networks, which are networks where jobs at a queue can have statistically non-identical routing behavior, are described in Section \ref{sec:multiclass}.)\

We first specify the processes that describe how jobs arrive and are processed through a network.
The number of external arrivals at different queues are assumed to be independent of each other, and IID at different times at each queue $j$, with mean $a_j\in \bR_+$.
The size of each job at queue $j$ is assumed to be  geometrically distributed with parameter $p_j$: for each unit of work assigned at a given time to queue $j$ by the schedule $\mS$, a job has probability $p_j$ of its service being completed; when this occurs, the remaining work at the queue is devoted to the next job at the queue whereas, if service is not completed by the assigned work, service of the job resumes at the next unit of time according to the allocation by the schedule at that time.
(An important special case, which we consider in our counterexamples, is where each job has unit size, i.e., $p_j=1$.)

After being served at queue $j$, a job moves to queue $j'$ with probability $P_{jj'}$ and, with probability $P_{j\omega}$, the job will leave the network, i.e., 
$
P_{j \omega} = {1-}\sum_{j'\in \mathcal J} P_{jj'}, \, j\in \mathcal J.
$
The matrix $\bm{P}=(P_{jj'} : j,j'\in \mathcal J)$ denotes the \emph{mean routing matrix}; we assume that $(I-P)$ is invertible, i.e., the network is \emph{open}.
A switched queueing network is a \emph{single-hop network} if each job leaves the network immediately after completing service, i.e., $P_{j\omega}=1$ for all $j$; otherwise, it is a \emph{{multihop} network}. In this paper we will be principally interested in {multihop} networks.

The queue size process ${\procQ}=(\vecQ(t): t\in \bZ_+)$, with  $\vecQ(t)=(Q_j(t): j\in\mJ)$, evolves according to the following \textit{queueing network equations}. For $t\in \mathbb N$ and $j \in \mathcal J$,
\begin{subequations}\label{eq:QN}
\begin{align}
Q_j(t) &= Q_j(0) + A_j(t) -D_j(t)\, , \label{eq:QN1}
\\[.2em]
A_j(t) &= E_j(t) +\sum_{j'\in \mathcal J}\Phi_{j'j}(D_{j'}(t))\, , \label{eq:QN2}
\\[.1em]
{D_j(t) } &= {S_j( \Pi_j(t))}   \, .\label{eq:QN3}
\end{align}
\end{subequations}
In the above equations,
 $A_j(t)$ is the \emph{cumulative number of arrivals} at queue $j$ (both external and internal),
 with the \emph{cumulative number of external arrivals} at queue $j$, $E_j(t)$, being the sum of the first $t$ external arrivals at $j$, and $\Phi_{jj'}(d)$ being the \emph{cumulative number of jobs} routed from queue $j$ to $j'$ after $d$ departures from queue $j$. It is assumed that $(\Phi_{jj'}(d)-\Phi_{jj'}(d-1) : j'\in \mJ\cup \{ \omega\})$, $d\in \mathbb N$,  are IID zero-one random vectors with the $j'$th-coordinate taking value $1$ with probability $P_{jj'}$ for $j' \in \mJ\cup \{ \omega\}$. 
The random variable $D_j(t)$ is the\emph{ cumulative number of departures} from queue $j$.
	One has $A_j(0)=E_j(0)=D_j(0)=0$.
 The \emph{cumulative service} scheduled at queue $j$,
\begin{equation}\label{eq:PI}
\Pi_j(t) := \sum_{\tau=1}^{t-1} {\Big(} \pi_j(\vecQ(\tau)) {\wedge Q_j(\tau) \Big)},
\end{equation}
where $\vecpi(\vecQ(\tau))$ depends on the policy.  For instance, $\vecpi(\vecQ)$ might be given by the MaxWeight schedule defined in \eqref{eq1} or the weighted MaxWeight schedule defined in \eqref{eq1wtmaxwt}.
 The random variable $S_j(u)$ is the \emph{cumulative number of jobs served} after $u\in\mathbb Z_+$ units of service are devoted to queue $j$. Since jobs sizes are geometrically distributed,  $S_j(u)$ is a discrete time renewal process whose inter-increment times are geometrically distributed with mean $1/p_j$.

The equation
\eqref{eq:QN1} states that the queue size is given by the arrivals minus the departures,
\eqref{eq:QN2} states that $A_j$ is the sum of both external and internal arrivals, and
\eqref{eq:QN3} states that $D_j$ is the number of jobs $S_j$ with completed service at $j$ after the amount of service $\Pi_j$. 

When considering the counterexample in Figure \ref{Fig2}, we will denote the total queue size at each of the two components by
\[
Q^{\Sigma}_{\mA }(t)
=
\sum_{j=0}^J Q_{\mA_j }(t)
\qquad
\text{and}
\qquad
Q^{\Sigma}_{\mB }(t)
=
\sum_{j=0}^J Q_{\mB_j }(t)\,.
\]

The order in which jobs arrive and jobs are served over a given time unit needs to be specified. In order to rule out instantaneous service of jobs, we assume that jobs initially present are first served, after which jobs arrive at their new queues.  The allocation of service at time $t$ is governed by $\vecQ(t-1)$ for the policies we will consider. 

Under the above assumptions, the queue size process $\procQ$ is a discrete time countable state space Markov chain.
This paper investigates the positive recurrence and transience of this Markov chain.


The \emph{total arrival rate} $\veclambda = (\lambda_j : j\in\mJ)$ is given by
\begin{equation*}
	\veclambda = \veca + \veca P + \veca P^2 + ... = \veca (I-P)^{-1}
\end{equation*}
or, alternatively, is the solution of the \emph{traffic equations}
\begin{equation}\label{eq:Traffic}
\lambda_j = a_j + \sum_{j'\in \mathcal J} \lambda_{j'} P_{j'j}\, ,
\qquad
j\in \mathcal J
 .
\end{equation}
Because the mean service time for each arrival is $p^{-1}_j$, the \emph{traffic intensity at queue $j$} is
\begin{equation}\label{rel.rho.lambda.p}
	\rho_j :=  \frac{\lambda_j}{p_j}\, .
\end{equation}
The \textit{traffic intensity} (i.e., nominal load) $\bm{\rho}=(\rho_j :j\in\mJ)$ gives the long-term rate at which work destined for each queue arrives in the queueing network. 

We recall from the introduction that $\mS \subset \bZ^{|\mJ|}_+$ denotes the set of feasible schedules and $<\mS>$ denotes the convex closure of the set $\mS$. The subcritical region $\mC$ denotes the interior of the set $<\mS >$, and a policy is maximally stable if the queue size process $\bm{\mathcal Q}$ is positive recurrent whenever $\vecrho\in \mC$.


We briefly motivate the definition of maximally stable as follows:
For any $\vecrho\in \mC$, there exists a randomized policy whose mean service rate $\bar{\vecpi}$ satisfies $\vecrho < \bar{\vecpi}$ and whose corresponding switched queueing network is positive recurrent.  So, stability is possible for a given switched queueing network policy when $\vecrho\in \mC$.  On the other hand, whenever $\vecrho$ lies strictly outside $\mC$, the total workload in the network must increase linearly over time, and so the switched queueing network must be transient.



\section{Proof of Theorem \ref{thrm:MW}.}
 \label{proof} 
 \label{sec:proofs}
In this section, we prove Theorem \ref{thrm:MW}. 
Since the proof is somewhat involved, we recommend that on a first reading the reader skip this section and continue to Section \ref{sec:stable}.
%
%

\subsection{Main Setup.}
\label{sec:Thrm1}
%
To prove Theorem \ref{thrm:MW}, we employ Lemma \ref{HighProb} and Proposition \ref{MainProp}, which are stated below.  The lemma is an elementary large deviations estimate; most of the work in this section will consist of demonstrating the proposition.   Assuming both the lemma and proposition, we demonstrate Theorem \ref{thrm:MW} in this subsection, postponing until Subsection \ref{sec:PropProof} the proofs of Lemma \ref{HighProb} and Proposition \ref{MainProp}.  We first introduce terminology.

%

Whenever \eqref{thm1condition} is satisfied,  $\gamma$ can be chosen so that
\begin{equation}\label{gamma.condition}
1 < \gamma < \frac{a}{1-a+a \nu/J} \ .
\end{equation}
(Note that the bound $J/(2J - \nu ) < a$ in  \eqref{thm1condition} implies $1< a/(1-a + a \nu/J)$ .) 
For $a$, $J$, $\nu$, and $\gamma$ fixed, we choose 
$\epsilon \in (0, \epsilon_0)$, where $\epsilon_0$ is a small positive constant.
We choose $M$ so that
	\begin{align}
	M  \geq L/\epsilon \,,  \label{lw.bound.M}
	\end{align}
	where $L$ is a large positive constant.  In the proofs below, we will also introduce positive constants $\kappa_1,...,\kappa_4$; 
	the constants $\epsilon_0$, $L$, $\kappa_1,...,\kappa_4$  
	will depend only on the parameters $a$, $J$, $\nu$, and $\gamma$. 

 We assume the following conditions on the initial queue lengths $\vecQ(0)$:
\begin{subequations}\label{InitalCondition}
	\begin{align}
	Q^{\Sigma}_{\mB }(0) & = M \, ,\label{InitialB}
\\[.2em]
	Q^{\Sigma}_{\mA }(0) & \leq \epsilon M / \nu \, ,\label{InitialA}
\\[.15em]
 \left| Q_{\mB_0}(0) - \nu Q_{\mB_j}(0) \right|  &\leq \epsilon M \, , \text{ for } j=1,...,J \ ,\label{Gap}
	\end{align}
\end{subequations}
The displays in  \eqref{InitalCondition} state  that there are initially $M$ jobs at Component $\mathcal B$
and comparatively few jobs at Component $\mathcal A$, 
with the queues at Component $\mathcal B$ being approximately ``balanced'' (that is, by (\ref{Sfortheorem1}), any of the queues of Component $\mB$ might receive immediate service).

Lemma \ref{HighProb} states that the numbers of arrivals at different queues are approximately deterministic over longer intervals of time.
\begin{lemma}\label{HighProb}
	For any $\delta > 0$, there exist $\eta>0$ and a constant $C$ such that, for all $T \geq C$,
		\[
		\mathbb P \Bigg(\sup_{0\leq s\leq t\leq T}\left| A_j(t)-A_j(s) - \frac{a}{J} (t-s) \right| \leq \delta {T} \Bigg) \geq 1 - e^{-\eta T}\, ,
		\]
	for  $j\in\{\mA_1,...,\mA_J , \mB_1,...,\mB_J\}$.
\end{lemma}

To prove Theorem \ref{thrm:MW}, 
we set
\begin{equation} \label{defofT}
T = 2M/(1-a)(1 - r_{\vecrho})
\end{equation}
and
$M' = \gamma M$,  
%
where $\gamma $ is assumed to satisfy (\ref{gamma.condition}), and assume $M$ to be sufficiently large so that $T \geq C$, where $C$ is as in Lemma \ref{HighProb}.
We designate by $G_M$ the non-exceptional set in Lemma \ref{HighProb} over all $ j\in\{\mA_1,...,\mA_J , \mB_1,...,\mB_J\}$, with $\delta$ replaced by $\epsilon b$:
\begin{equation}\label{Event}
G_M 
=
\left\{
	\sup_{ j\in\{\mA_1,...,\mA_J , \mB_1,...,\mB_J\} }
	\sup_{0\leq s\leq t\leq T}\left| A_j(t)-A_j(s) - {\frac{a}{J}} (t-s) \right| \leq \epsilon\, b   \, T
\right\}\, ,
\end{equation}
where $ b = (\gamma - 1)(1-a)(1 - r_{\vecrho})/6 \nu  J$. ($b$ is chosen for convenience, but needs to be small.) 
It follows immediately from Lemma \ref{HighProb} that
\begin{equation}  \label{GMbound}
\mathbb{P}(G_M)  \geq 1 - e^{-\eta T}\,,
\end{equation}
for a new choice of $\eta > 0$.

Proposition \ref{MainProp} states that, on $G_M$, there exists a stopping time $V \le T$ 
such that, at $V$, an approximate analog of the initial condition \eqref{InitalCondition} holds,  but with $M$ replaced by $M'$ and the roles of the components $\mathcal A$ and $\mathcal B$ reversed. In particular, the amount of work in the queueing network has increased by a multiplicative factor $\gamma$.

\begin{proposition}\label{MainProp} 
Assume (\ref{superqueueload}) and (\ref{thm1condition}), and assume \eqref{InitalCondition} with $M$ satisfying \eqref{lw.bound.M}.  Then, on the event $G_M$, 
	\begin{subequations}\label{Init}
		\begin{gather}
		Q^{\Sigma}_\mA  (V) \geq M',\label{Init1}
\\[.2em]
		Q^{\Sigma}_\mB  (V) \leq \epsilon M' / \nu , \label{Init2} 
\\[.12em]
	    \left| Q_{\mA_0}(V) - \nu Q_{\mA_j} (V) \right| \leq \epsilon M' ,\label{Init3}
		\end{gather}
\end{subequations}
for $j=1,...,J$, where $V$ is a stopping time satisfying $0 < V \leq T$ that will be defined in (\ref{eq:tildeT}), and $M'=\gamma M$ is defined above.  Moreover, for all times $0\leq t\leq V$,
	\begin{equation}\label{Init4}
	Q_\mA^\Sigma(t) + Q_\mB ^\Sigma(t) \geq \frac{1}{2} \frac{a}{a +\nu}M.
	\end{equation}
\end{proposition}

Both Lemma \ref{HighProb} and Proposition \ref{MainProp} are proved in Subsection \ref{sec:PropProof}.

The proof of Theorem \ref{thrm:MW}  is a straightforward consequence of (\ref{GMbound}) and Proposition \ref{MainProp}.

\smallskip

\begin{proof}[Proof of Theorem \ref{thrm:MW}.]
 $\vecrho \in \mathcal C$ follows from (\ref{superqueueload}). 

To show $\vecQ(t)$ is transient,
note that,
since $a,\nu$ and $J$ satisfy \eqref{thm1condition},
$\gamma$ can be selected so that \eqref{gamma.condition} holds. The constants $a,\nu$, $J$ and $\gamma$ determine $\epsilon_0$, $L$, and $\kappa_1,...,\kappa_4$, and $\epsilon$ and $M$ are selected so that $\epsilon \in (0, \epsilon_0)$
 and $M \geq L / \epsilon$.

Denote by $\mathcal I_M$ and $\mathcal I_{M'}$  the sets of vectors with queue lengths satisfying \eqref{InitalCondition} and \eqref{Init}, respectively.
	By  (\ref{GMbound}) and Proposition \ref{MainProp}, for any initial state 
	$\mathbf Q_0 := (Q_{\mA_0}(0),..., Q_{\mA_J} (0),Q_{\mB_0} (0),...,Q_{\mB_J} (0)) \in \mathcal I_M$, there exists a stopping time $V\leq T$ with
	\[
	\mathbb P_{\mathbf Q_0}
	\left(
		(Q_\mA (V),Q_\mB (V))
		\in \mathcal I_{ M'} \,,
	\,\,\,
	\min_{0 \leq t \leq V}\Big\{ Q_\mA ^\Sigma(t) + Q_\mB ^\Sigma(t) \Big\} \geq  \frac{1}{2} \frac{a}{a + \nu}M
	\right)
	\geq 1 - e^{-\eta T} .
	\]
 Iteration using the Strong Markov property implies that
	\[
	\mathbb P_{\mathbf Q_0} \left(   Q_\mA ^\Sigma(t) + Q_\mB ^\Sigma(t) <  \frac{1}{2} \frac{a}{a + \nu} M \quad \text{for some }t\ge 0 \right)
	\leq
	\sum_{k=0}^\infty e^{-\eta \gamma^k T}\, .
	\]
	(More careful iteration in fact shows that $Q_\mA ^\Sigma(t) + Q_\mB ^\Sigma(t) \rightarrow \infty$ as $t \rightarrow \infty$ off of the above exceptional set.)
	
Since the right-hand side of the last inequality is less than $1$ for sufficiently large $M$, there is positive probability that the queueing system is never empty when starting from any state in $\mathcal I_M$, such as in \eqref{InitalCondition}.
Since all states communicate, this implies the process $\vecQ(t)$ is transient.
\end{proof}

\subsection{Proofs of Lemma \ref{HighProb} and  Proposition \ref{MainProp}.}\label{sec:PropProof}  The proof of Lemma \ref{HighProb} employs elementary large deviation bounds.

\smallskip

\begin{proof}[Proof of Lemma \ref{HighProb}.]
For any $j\in \{ \mathcal A_1,..., \mathcal A_J, \mathcal B_1,..., \mathcal B_J \}$,
\begin{align} \label{expbound}
	\mathbb P \left(
	\sup_{0\leq s\leq t\leq T}
	\left| 
	A_j(t)-A_j(s) - \frac{a}{J} (t-s) 
	\right|
	> \delta T  
	\right) 
	& 
	\leq 
	\mathbb P \left( 
	\sup_{0\leq t\leq T}
	\left| 
	A_j(t) - \frac{a}{J} t
	\right| 
	> \delta T /2   
	\right)   \notag
	\\[.1em]
& 
	\leq \sum_{t=0}^T
	\mathbb P \left( 
	\left| 
	A_j(t) - \frac{a}{J} t
	\right| 
	> \delta T /2   
	\right) \,.  
\end{align}

Each term $A_j(t) - a t/J$ is the sum of $t$ i.i.d. random variables $X_i$, $i=1,\dots, t$, with mean $0$ and $\bE e^{\theta X_i}<\infty$, for $\theta > 0$.  
Setting $\Lambda(\theta)=\log \bE [ \exp\{\theta X\}]$, we note that $\Lambda(0)=0$ and $\Lambda'(0)=\mathbb E X=0$. So, $\alpha:=\delta \theta /2 - \Lambda(\theta) > 0$ for small enough $\theta >0$.  Choosing one such value of $\theta$ and using Markov's Inequality, it follows that
\begin{align} \label{Markov} \notag
\mathbb P\left(\sum_{i=0}^t X_i  > \delta T/2\right) 
& \le e^{  \delta \theta (t-T)/2} \bE \exp \left\{\theta \sum_{i=0}^t (X_i -\delta /2) \right\}  
\\[.2em]
& = e^{  \delta \theta (t-T)/2} \left( \bE e^{\theta (X_1 -\delta /2)}\right)^t 
= e^{  \delta \theta (t-T)/2} e^{-\alpha t}  \le e^{-\eta T}
\end{align}
for $\eta = \min\{\delta \theta /2 , \alpha\} >0$.	 The same argument and (\ref{Markov}) 
continue to hold when $X_i$ is replaced by $-X_i$.
Applying these two bounds to each term on the right-hand side of (\ref{expbound}), it follows that the left-hand side of  (\ref{expbound}) is at most $2T e^{-\eta T}$, and the lemma follows with a new choice of $\eta$.
\end{proof}

\smallskip


We now begin the demonstration of Proposition \ref{MainProp}.  
Set
\begin{equation}\label{eq:tildeT}
U
=
\min\big\{ t \geq 0 : Q_{\mB}^{\Sigma}(t)\leq \nu^2 \big\} \wedge T \,, \quad
V
=
\min
\big\{
t \geq U : Q_{\mA_0}(t) \le \, \nu \! \! \max_{j=1,...,J} Q_{\mA_j}(t)  
\big\} \wedge T \ ,
\end{equation}
where $T$ is chosen as above.
The stopping time $U$ can be thought of the first time at which Component $\mB$ ``is almost empty"; $V$ is the first time after this
 that a queue, other than $\mA_0$, might be served at Component $\mA$.  

The demonstration of (\ref{Init1})--(\ref{Init3}) and (\ref{Init4}) of Proposition  \ref{MainProp} employs six lemmas that track the flow of jobs through the network over the time interval $[0,V]$, on the event $G_M$.  The main steps are  as follows.

\smallskip

\begin{itemize}
\item 
 $U < T$ 
(Lemma \ref{LemmaEmpty}):
The number of jobs in Component $\mB$ at any time is bounded above by the sum of the number of jobs originally in Component $\mB$ and Queues $\mA_j$, $j=1,\ldots,J$, together with the number of arrivals in $\mA_j$,  $j=1,
\ldots,J$, up to this time, less the number of departures from $\mB$ by then.  Since the component $\mB$ is subcritical,  $U$ must occur by a computable time 
before $T$.

\smallskip

\item  
Up until time $U$, jobs arrive at  Queue $\mA_0$ from Queues $\mB_j$, $j=1,\dots,J$, at a rapid linear rate (Lemma \ref{A0 arrivals}):
Because Component $\mB$ is initially balanced and, by the MaxWeight policy, remains balanced, a fixed fraction of the service at Component $\mB$ occurs at the ``quick service" Queues $\mB_j$, $j=1,\ldots,J$.

\smallskip

\item 
For appropriate $\kappa_2 > 0$, 
$Q_{\mA_0} (t) > \nu \max_{j=1,...,J} Q_{\mA_j}(t)$
on $[U \wedge \epsilon \kappa_2   M , V)$ (Lemma \ref{LemmaStopService}): The inequality is immediate on $[U,V)$ from the definition of $V$, but we require the stronger lower bound.  This part employs Lemma \ref{A0 arrivals}, together with the lower rate at which jobs arrive at $\mA_j$, $j=1,\ldots,J$, from outside the network.
Note that the MaxWeight policy, together with the above inequality,  implies that, on $(U \wedge \epsilon \kappa_2   M , V]$, the Queues $\mA_j$, $j=1,\dots, J$ do not receive service.

\smallskip

\item
For all $t\in [0,T]$ and $j,j'=1,\ldots,J$, $|Q_{\mA_j}(t) - Q_{\mA_{j'}}(t) |$ is small (Lemma \ref{Lemon9}):  This is an elementary consequence of the MaxWeight policy.

\smallskip

\item
During $[U,V]$, there are few jobs in Component $\mB$ (Lemma \ref{LemmaStillEmpty}):  This is the case at time $U$.   Because of the MaxWeight policy and Lemma \ref{LemmaStopService}, there are no arrivals at Queue $\mB_0$ over $(U,V]$.  Moreover, the ``quick service" queues $\mB_j$, $j=1,\ldots,J$, serve jobs much faster than they arrive.  Note that (\ref{Init2})  of Proposition \ref{MainProp} is an immediate consequence of Lemma \ref{LemmaStillEmpty}.

\smallskip

\item
$V <T$ ((\ref{eq:Lem8A}) of Lemma \ref{Lemma8}):  
On $(U,V]$, Queue $\mA_0$ receives all of the service at Component $\mA$, which is greater than the rate at which jobs arrive along the route leading to $\mA_0$.  But the number of jobs at the other queues in $\mA$ is increasing linearly because of external arrivals.  Comparison gives an explicit upper bound on $V$ that implies $V < T$.  Note that $|Q_{\mA_0} (V) - \nu \max_{j=1,...,J} Q_{\mA_j}(V)|\le \nu$ is immediate from $V<T$ and the definition of $V$.  It follows from this and Lemma \ref{Lemon9} that (\ref{Init3}) of Proposition \ref{MainProp} is  satisfied.

\smallskip

\item
$\mQ^{\Sigma}_A(V) \ge M'$, where $M'$ is as in
Proposition \ref{MainProp} ((\ref{eq:Lem8B}) of  Lemma \ref{Lemma8}):
One can show that the above upper bound on $V$ is also the lower bound, up to a lower order term.  The linear growth in the number of jobs at the Queues $\mA_j$, $j=1.\ldots,J$, over times $t \in (\epsilon \kappa_2 M,V]$, together with $Q_{\mA_0} (V) = \nu \max_{j=1,...,J} Q_{\mA_j}(V)$, implies $\mQ^{\Sigma}_A(V) \ge M'$, which  is the claim in (\ref{Init1}) of Proposition \ref{MainProp}.  

\smallskip

\item
For all $t\in [0,V]$, the number of jobs in the network is at least $aM/2(a +\nu)$, i.e., (\ref{Init4}) of Proposition \ref{MainProp} holds (proof of  Proposition \ref{MainProp}):  The number of jobs at the Queues $\mA_j$, $j=1,\dots,J$ increases linearly over $(\epsilon \kappa_2 M, V]$.  But the number of jobs at Component $\mB$ does not decrease from its initial value $M$ at more than an explicit linear rate.  The minimum of the sum will be at least
 $aM/2(a +\nu)$.
\end{itemize}

\begin{figure}
	\centering
	\includegraphics[width=\textwidth]{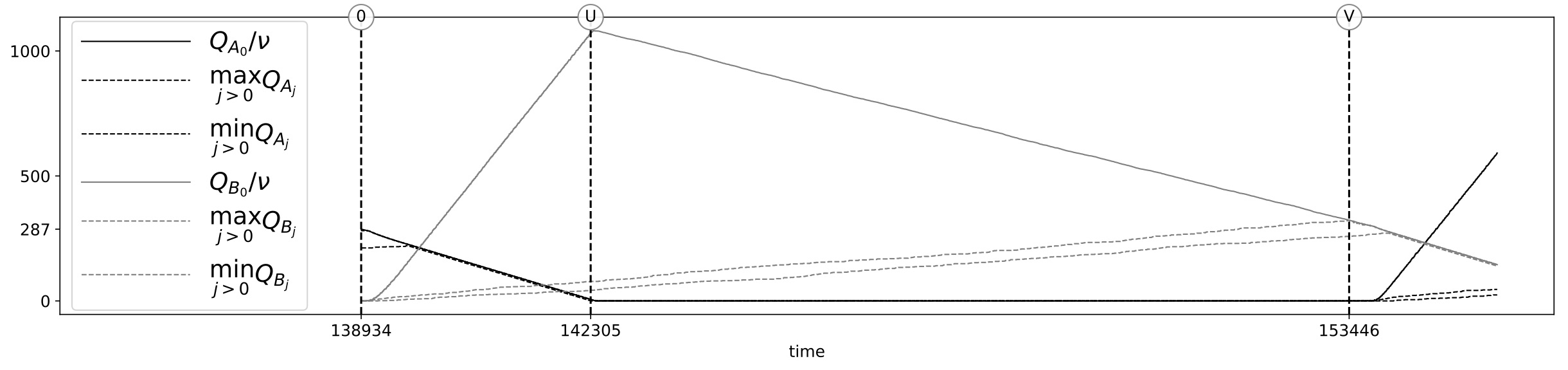}
	\caption{A snapshot of the network simulated in
Figure \ref{Fig:Sim1}
(with parameters
	$a=7/12$, $\nu=6$, and $J=30$), starting at time $138,934$ and labeled by $0$ at the top of the diagram. The stopping times $U$ and $V$, in  Proposition \ref{MainProp} and the accompanying outline of its proof, correspond to times $142,305$ and $153,446$ in the simulation in Figure \ref{Fig:Sim1}.  Note that, by (\ref{gamma.condition}), Proposition \ref{MainProp} is valid for $\gamma \in (1, 35/32)$, and the growth in the total number of jobs from time $138,934$ to time $153,446$ in Figure \ref{Fig:Sim1} and the diagram is relatively small.
	\label{Fig:Sim1Cut}}
\end{figure}


%



\medskip

Figure \ref{Fig:Sim1Cut} illustrates the evolution over a single cycle of the switched queueing network in Figure \ref{Fig:Sim1}}.  The relative sizes of the queues at and between times $0$, $U$, and $V$ can be compared with the bounds given in the above outline.

We first demonstrate Lemma \ref{LemmaEmpty}.

\begin{lemma}\label{LemmaEmpty}
Assume (\ref{superqueueload}), and  assume \eqref{InitalCondition} with $M$ satisfying \eqref{lw.bound.M}.
On the event $G_M$,
\begin{equation}\label{T.tilde.bounds}
 U
 \leq
 \frac{
 	2 M
 }{
 1-a(1+{1}/{\nu})
 }\, .
\end{equation}
In particular, $U < T$.
\end{lemma}

\begin{proof}
 On $t\le U$, $Q^{\Sigma}_{\mB}(t) > 0$, and so
	$D_{\mB_0}(t) + \frac{1}{\nu}\sum_{j=1}^J D_{\mB_j}(t) = t$.
On the other hand, arrivals at Queue $\mB_0$ are due to departures from the Queues $\mA_j$, and so
\begin{equation} \label{displayinlemma2}
A_{\mB_0}(t) 
=
\sum_{j=1}^J D_{\mathcal A_j}(t)
	\leq \sum_{j=1}^J Q_{\mA_j}(0) + \sum_{j=1}^J  A_{\mA_j}(t) 
	\leq \epsilon \frac{M}{\nu}  + at + \epsilon b  J  T  
\quad \text{ on } G_M \, ,
\end{equation}
with \eqref{InitialA} and (\ref{Event}) being employed in (\ref{displayinlemma2}). 
It follows from \eqref{InitalCondition} and the above two inequalities that, on $G_M$, for $t\le U$,
\begin{align}
\frac{1}{\nu} Q^{\Sigma}_{\mB}(t)
&\leq
	Q_{\mB_0}(t) + \frac{1}{\nu} \sum_{j=1}^J Q_{\mB_j}(t)
\leq 
	M + A_{\mB_0}(t) +  \frac{1}{\nu} \sum_{j=1}^J A_{\mB_j}(t) - t \notag 
\\[.2em]
&\leq 
	M 
	+ \Big( at + \epsilon b J  T\Big)
	+ \frac{J}{\nu} \Big(\frac{a}{J} t + \epsilon b  T\Big) - t \, . \label{eq:QBupper}
\end{align}
Setting $t = U$ and noting that $Q^{\Sigma}_{\mB}(U) \ge 0$, it follows from (\ref{eq:QBupper}) that
\begin{equation}\label{constraint.lemma.3}
U
\leq
	\frac{(1+{\epsilon}/{\nu})M + (1+{1}/{\nu})  \epsilon b   J  T }{1-a(1+{1}/{\nu})} 
\end{equation}
since, by (\ref{superqueueload}), $a(1+{1}/{\nu}) < 1$.  For small enough $\epsilon > 0$,  
the numerator of (\ref{constraint.lemma.3}) is less that $2M$, which implies (\ref{T.tilde.bounds}). 
 Also by (\ref{superqueueload}), the denominator equals $1 - r_{\vecrho} > 0$, and so $U < T$ follows from
$T = 2M/(1-a)(1 - r_{\vecrho})$.
\end{proof}

\medskip

Lemma \ref{A0 arrivals} shows that, for $t\le U$, the number of arrivals at {Queue} $\mA_0$ is bounded below by $(1-a)\nu t$, up to smaller order terms.  

\begin{lemma}\label{A0 arrivals}  
Assume (\ref{superqueueload}), and assume \eqref{InitalCondition} with $M$ satisfying \eqref{lw.bound.M}.
On the event $G_M$,
there exists a constant $\kappa_1 > 0$ such that, for $t\leq U$,
\begin{equation} \label{lemma3eq}
A_{\mA_0} (t)
\geq\nu (1 - a ) \left(1 - 
  \frac{\nu^2}{\nu^2+J}  \right)t
- \epsilon  \kappa_1 M  \, .
\end{equation}
\end{lemma}
\begin{proof}
We first demonstrate the inequalities (\ref{eq:Service}), (\ref{step1Delta}), and (\ref{longline}),  which will be used to show (\ref{longinequality}).  The inequality (\ref{longinequality}) will then be rearranged to complete the proof.

We note that, since Component $\mathcal B$ is not empty over $[0,U]$ and the departure of a job at any Queue $\mB_1,..., \mB_J$ corresponds to an arrival at Queue $\mA_0$, 
\begin{equation}
\label{eq:Service}
D_{\mB_0}(t) + \frac{1}{\nu} A_{\mA_0}(t) 
= D_{\mB_0}(t) + \frac{1}{\nu} \sum_{j=1}^J D_{\mB_j}(t)
= t \ , \quad  \text{for } t \leq U\,.
\end{equation}


Let $S$ be the last time up to time $t$ at which the Queue $\mB_0$ is served, that is, 
$S = \max\{u \leq t : D_{\mB_0}(u+1) = 1 + D_{\mB_0}(u) \}$.
By the MaxWeight policy, 
\begin{equation}\label{eq.upp.max}
	\max_{j=1,...,J} \nu Q_{\mB_j}(S) 
	\leq Q_{\mB_0}(S) 
	\leq Q_{\mB_0}(S+1) + 1 
	\leq Q_{\mB_0}(t) + 1\ ,
\end{equation}
with the last inequality following from the definition of $S$.  Consequently,
\begin{equation}\label{step1Delta}
	\Delta := \min_{j=1,...,J} Q_{\mB_j}(0) - \max_{j=1,...,J} Q_{\mB_j}(S) \geq  \min_{j=1,...,J} Q_{\mB_j}(0) - \frac{1}{\nu} (Q_{\mB_0}(t) + 1) \ .
\end{equation}

The inequality 
\begin{equation}
\label{longline}
  A_{\mB_0}(t)
  =  \sum_{j=1}^J D_{\mA_j}(t)
  \leq Q^{\Sigma}_{\mA}(0) + \sum_{j=1}^J A_{\mA_j}(t) 
  \leq \epsilon \frac{M}{\nu} + a t + \epsilon  b J  T  
\end{equation}
follows by applying 
 \eqref{eq:QN1}, and then the initial condition \eqref{InitialA} together with our restriction to the event $G_M$.

Employing the above inequalities, we obtain
\begin{align}
\label{longinequality}
\notag
A_{\mA_0}(t)
&\geq J \Delta \\  \notag
  & \geq \frac{J}{\nu}  \Big[  
  \min_{j=1,...,J} \nu Q_{\mB_j}(0)
  - Q_{\mB_0}(0)
  +Q_{\mB_0}(0)
   - Q_{\mB_0}(t) - 1 \Big] \\
  &\geq\frac{J}{\nu} \Big[ 
  -\epsilon M
   - A_{\mB_0}(t) + D_{\mB_0}(t) - 1 \Big] 
\notag 
\\[.25em]
  &\geq \frac{J}{\nu} \Bigg[ -\epsilon M - \Big( \epsilon \frac{M}{\nu} + a t + \epsilon b J   T \Big) + \bigg( t - \frac{A_{\mA_0}(t)}{\nu} \bigg) -1  \Bigg] \, .
\end{align}
The first inequality in (\ref{longinequality}) follows from 
the observation that at least $\Delta$ jobs have been served by time $S\le t$ at each of the $J$ queues in Component $\mB$, the second inequality follows from
\eqref{step1Delta},
the third inequality follows from \eqref{eq:QN1} and \eqref{Gap}, and the
 fourth inequality follows from \eqref{eq:Service} and (\ref{longline}).
 
Notice by together collecting smaller order terms, we can write \eqref{longinequality} as a follows
 \[
 A_{\mathcal A_0} (t) 
 \geq 
 \frac{J}{\nu}
 \left[
 (1-a)t 
 - \frac{A_{\mathcal A_0}(t)}{\nu}
 - \epsilon \kappa M
 \right] \, ,
 \]
where $\kappa >0$ is a constant that does not depend on $\epsilon$ or $M$. (Recall that $T = 2M/(1-a)(1 - r_{\vecrho})$.) Combining the terms for $A_{\mathcal A_0}(t)$ gives 
\begin{align*}
 A_{\mathcal A_0} (t)  
&
 \geq 
 \frac{\nu J}{\nu^2 +J}
 \left[
 (1-a) t - \epsilon \kappa M
 \right]
\\
&
=\nu (1-a) t \left( 1 - \frac{\nu^2}{\nu^2 + J } \right) - \epsilon \kappa_1 M 	\, ,
\end{align*}
where $\kappa_1 =  \nu J \kappa /(\nu^2 +J)$. This gives the required bound in $A_{\mathcal A_0}(t)$.
\end{proof}

\medskip

The following lemma asserts that, over a time interval of interest to us, only Queue $\mathcal A_0$ is served at Component $\mathcal A$. 
\begin{lemma} \label{LemmaStopService} 
Assume (\ref{superqueueload}) and  \eqref{thm1condition}, and assume  \eqref{InitalCondition} with $M$ satisfying \eqref{lw.bound.M}.
On the event $G_M$,
there exists a sufficiently large $\kappa_2$ such that,
 for $t \in [U \wedge \epsilon \kappa_2   M , V)$,
	\begin{equation} \label{displayforlemm4}
	Q_{\mA_0} (t)
	>
	\nu \max_{j=1,...,J} Q_{\mA_j}(t)\, .
	\end{equation}
In particular, on $(U \wedge \epsilon \kappa_2   M , V]$,   only Queue $\mathcal A_0$ is served at Component $\mathcal A$.
\end{lemma}

Recall that service at time $t$ is governed by $\vecQ(t-1)$.

\smallskip

\begin{proof}[Proof of Lemma \ref{LemmaStopService}.]  The last statement in the lemma follows from (\ref{displayforlemm4}) and the MaxWeight policy.  

To prove (\ref{displayforlemm4}), note that (\ref{displayforlemm4}) follows automatically from the definition of $V$ for $t\in [U,V)$, and so it suffices to demonstrate (\ref{displayforlemm4}) for $t\in [\epsilon \kappa_2 M, U)$.   First note that, by
		Lemma \ref{A0 arrivals},
		\[
A_{\mA_0} (t)
\geq\nu (1 - a ) \left(1 - 
  \frac{\nu^2}{\nu^2+J}  \right)t
- \epsilon  \kappa_1 M \quad	\text{for all } t\leq U	\,.
\]
Since
$
		D_{\mA_0}(t) \leq  t  \, ,
$
it follows that
\begin{equation}\label{LemBound1}
	Q_{\mA_0}(t)
	=
	Q_{\mA_0}(0)
	+
	A_{\mA_0}(t)
	-
	D_{\mA_0}(t)
\geq
	\left(\nu(1-a)   -
\nu(1-a) \frac{\nu^2}{\nu^2+J} -1\, 
\right) t  - \epsilon\kappa_1 M\,,
\end{equation}
for all $t\leq U$.  On the other hand, by \eqref{InitialA},  on the event $G_M$,
\begin{equation}\label{LemBound2}
	\nu Q_{\mA_j}(t) \leq (a \nu /J)t + \epsilon (M + b \nu T)\,  ,	
\end{equation}
for $ j=1,...,J$ 
	 and  $t \leq T $.

To demonstrate (\ref{displayforlemm4}), it therefore suffices to show that
 the right-hand side of \eqref{LemBound1} is greater than the right-hand side of \eqref{LemBound2} for $t\in [\epsilon \kappa_2 M, U)$, for $\kappa_2$ chosen sufficiently large.  It suffices to compare the coefficients of $t$ in these two expressions, that is, to show
\begin{equation} \label{equationfor_a}
\nu(1-a)    -
\nu(1-a) \frac{\nu^2}{\nu^2+J} -1 \, 
> \frac{a \nu }{J}\, .
\end{equation}
After some calculation, one can check that (\ref{equationfor_a}) is equivalent to the upper bound for $a$ in (\ref{thm1condition}), which completes the proof of the lemma.  (With additional work, one can show that $\epsilon \kappa_2M \le U$ on $G_M$, by properly quantifying $\kappa$, $\kappa_1$, and $\kappa_2$.)
\end{proof}

\medskip

 Lemma \ref{Lemon9} bounds the queue size differences for $\mA_1$,...,$\mA_J$ up to time $V$.  Because of the MaxWeight optimization, this difference will be 
 bounded above by the corresponding difference in arrivals, which is elementary to bound.

\begin{lemma}\label{Lemon9}
Assume \eqref{InitalCondition}.  On the event $G_M$, 
\[
\max_{j,j'=1,...,J}   \max_{0\leq t \leq T}|Q_{\mA_j}(t) - Q_{\mA_{j'}}(t) |
\leq
\epsilon M/\nu + 3\epsilon bT \, .  
\]
\end{lemma}

\begin{proof}
We show that the above bound holds at any time $t$ and any pair of queues, $\mA_j$ and $\mathcal A_{j'}$. 
Without loss of generality, assume $Q_{\mA_j}(t) > Q_{\mA_{j'}}(t)$. 
Let $\tau$ denote the last time before time $t$ such that $Q_{\mA_j}(\tau) \leq  Q_{\mA_{j'}}(\tau)$; if no such time exists, set $\tau=0$. Then
\begin{equation}\label{Q_Tau}
	Q_{\mA_j}(\tau+1) - Q_{\mA_{j'}}(\tau+1) \leq  \nu + 1 + \epsilon M/\nu\, ,
\end{equation}
where the term $\epsilon M/ \nu$ is due to the initial condition \eqref{InitialA}, for the case where $\tau = 0$.

From time $\tau+1$ until time $t$, Queue $\mA_j$ has strictly more jobs than $\mA_{j'}$, and so only Queue $\mA_j$ can receive service. Therefore,
\begin{align}\label{Q_st}
Q_{\mA_j}(t) &\leq Q_{\mA_j}(\tau+1)+ A_{\mA_j}(t)-A_{\mA_{j}}(\tau+1) \,,
\\[.2em]
Q_{\mA_{j'}}(t) &= Q_{\mA_{j'}}(\tau+1)+ A_{\mA_{j'}}(t)-A_{\mA_{j'}}(\tau+1)\,. \nonumber
\end{align}
Applying \eqref{Event}, \eqref{Q_Tau}, and \eqref{Q_st}, it follows that 
\begin{align*}
	&
	\left| Q_{\mA_j}(t) - Q_{\mA_{j'}}(t) \right|  \\
	 &
	\le Q_{\mA_j}(\tau+1) - Q_{\mA_{j'}}(\tau+1)
	+ \Big| A_{\mA_j}(t)-A_{\mA_{j}}(\tau+1)\Big|
	+ \Big| A_{\mA_{j'}}(t)-A_{\mA_{j'}}(\tau+1)  \Big| 
\\[.1em]
	& \le
\nu +1 + \epsilon M/\nu + 2 \epsilon  b T \le \epsilon M/\nu + 3\epsilon bT
\end{align*}
for large $T$, from which the lemma follows.
\end{proof}

\medskip

Lemma \ref{LemmaStillEmpty} shows that Component $\mB$ is approximately empty for $t \in [U,V]$.

\begin{lemma}\label{LemmaStillEmpty}
Assume (\ref{superqueueload}) and \eqref{thm1condition}, and assume \eqref{InitalCondition} with $M$ satisfying \eqref{lw.bound.M}. On the event $G_M$,
\begin{equation}
	Q^{\Sigma}_\mB (t)
	\leq
	\nu^3  +  \epsilon bJ T
\quad \text{ for } t \in [U,V]
	\ . \label{QBSum.upper.bound}
\end{equation}
\end{lemma}

\begin{proof}
By definition, $Q_{\mB_0}(U) \le Q_{\mB}^{\Sigma}(U) \le \nu^2$. On the other hand, by Lemma \ref{LemmaStopService}, Queues $\mA_1,...,\mA_J$ receive no service over the time interval $(\epsilon \kappa_2  M, V]$, and so
Queue $\mB_0$ has no arrivals then.
Hence, with the exception of at most $\nu^2$ units of service at Queue $\mB_0$, all service at component $\mB      $ over $(\epsilon \kappa_2  M, V]$ is devoted to  Queues $\mB_1,...,\mB_J$.

Suppose that $Q^{\Sigma}_{\mB}(u)\leq \nu^2$
at a given time $u \in [\epsilon \kappa_2  M, V)$.
	On $G_M$, the arrivals to Queues $\mathcal B_1, ..., \mathcal B_J$ are bounded as in \eqref{Event}, and jobs are served at rate $\nu$ there. 
Together with the previous paragraph, this implies that, for any $s \le V-u$ chosen so that
$Q^{\Sigma}_{\mB} (t) \ge \nu$ for all $t \in [u,u+s)$,
\begin {equation} \label{boundonQsigma}
Q^{\Sigma}_{\mB}
(u+s)
\leq
\nu^2 + (a - \nu)s + (\nu - 1) \nu^2 + \epsilon b  JT\, 
\leq 
\nu^3  + \epsilon b  JT\,,
\end{equation}
since $a < \nu$.
(The $(\nu - 1) \nu^2$ term accounts for the at most $\nu^2$ units of service that are devoted to serving Queue $\mB_0$ rather than Queues $\mathcal B_1, ..., \mathcal B_J$.)
Since $Q^{\Sigma}_{\mB} (U) \leq \nu^2$,  (\ref{QBSum.upper.bound}) follows by applying  (\ref{boundonQsigma}) whenever $Q^{\Sigma}_{\mB}(u) \in (\nu (\nu - 1), \nu^2]$ occurs.
\end{proof}
\medskip

Lemma \ref{Lemma8} provides explicit upper and lower  bounds on the time $V$ and a lower bound on the total queue size at Component $\mA$.  The upper bound on $V$ implies that $V < T$.

\begin{lemma}\label{Lemma8}
Assume (\ref{superqueueload}) and \eqref{thm1condition}, and
assume \eqref{InitalCondition} with $M$ satisfying \eqref{lw.bound.M}.
On the event $G_M$,
 there exist constants $\kappa_3 >0$ and $\kappa_4 >0$ satisfying
\begin{align}
&
\left|
	V-\frac{JM}{J+\nu} \cdot \frac{1}{1 - a + {a\nu }/{J} }
\right|
\leq
\epsilon \kappa_3   M \, ,\label{eq:Lem8A}
\\[.5em]
&
Q^{\Sigma}_{\mA}(V)
\geq
 \frac{aM}{1 -a + {a\nu}/{  J}}
-
\epsilon  \kappa_4  M\,. \label{eq:Lem8B}
\end{align}
By (\ref{eq:Lem8A}), $V < T$ for sufficiently small $\epsilon >0$.  
\end{lemma}

We remark that $V < T$, with Lemma \ref{LemmaStopService}, implies that $|Q_{\mA_0} (V) - \nu \max_{j=1,...,J} Q_{\mA_j}(V)| \le \nu$.

\smallskip

\begin{proof}[Proof of Lemma \ref{Lemma8}.]
We first demonstrate \eqref{eq:Lem8A}, which requires most of the work. Since all arrivals at Queue $\mathcal A_0$ are from Queues $\mathcal B_1,..., \mathcal B_J$,
$
Q_{\mathcal A_0}(t)
=
Q_{\mathcal A_0}(0)
+
\sum_{j=1}^J D_{\mB_j}(t)
-
D_{\mathcal A_0}(t)
$, and consequently,
\begin{equation}
\label{eq:QA5}
Q_{\mathcal A_0}(t)
=
Q_{\mathcal A_0}(0)
+
\sum_{j=1}^J Q_{\mB_j}(0) 
-
\sum_{j=1}^J 
Q_{\mathcal B_j}(t)
+
\sum_{j=1}^J
A_{\mB_j}(t)
-
D_{\mathcal A_0}(t) \ .
\end{equation}

We proceed to bound the five terms on the right-hand side of (\ref{eq:QA5}).  By \eqref{InitialA}, $Q_{\mA_0}(0) \leq \epsilon M /\nu$, and by 
\eqref{InitialB} and \eqref{Gap},
	\[
	\Big| \sum_{j=1}^J Q_{\mB_j}(0) - JM/ (J+\nu)\Big|
	 \leq
 \epsilon JM\, .
	\]
By Lemma \ref{LemmaStillEmpty} and \eqref{Event}, respectively, on $G_M$,
\[
\sum_{j=1}^J 
Q_{\mB_j} (t)
\leq 
\nu^3  + \epsilon b  JT\,, \quad \Big|\sum_{j=1}^J A_{\mB_j} (t) - at \Big|
	\leq
	 \epsilon b  J T \qquad  \text{ for } t\in [U,V] \,.
\]
Also, by Lemma \ref{LemmaStopService}, all service at Component $\mA$ over $(U \wedge \epsilon \kappa_2 M, V]$ occurs at Queue $\mA_0$, and so
\[
		t-  \epsilon \kappa_2   M \leq D_{\mA_0}(t) \leq t
\qquad  \text{ for } t\in [U,V]   \ .
	\]

Applying the above inequalities to \eqref{eq:QA5} implies the following bounds on $Q_{\mA_0}(t)$, for $ t\in [U,V]$:
\begin{equation}\label{eq:QABounds_a}
Q_{\mA_0}(t)
\leq
 \epsilon M/\nu  
+
\left[ JM/(J+\nu)  + \epsilon J M \right]
-
0
+
\left[ at + \epsilon b  {J} {T}\right]
-
\left[ t- \epsilon \kappa_2   M \right] \, ,
\end{equation}
and
\begin{equation}  \label{eq:QABounds_b}
Q_{\mA_0}(t)
\geq
 0  
+
\left[ JM/(J+\nu)  - \epsilon J M \right]
-
[
\nu^3   +  \epsilon b JT
]
+
\left[ at - \epsilon b  {J} {T}\right]
-
 t \ .
\end{equation}

On the other hand, by Lemma \ref{LemmaStopService}, the Queues $\mA_1,...,\mA_J$ will not be served over $(\epsilon \kappa_2   M,V]$, and so $D_{\mA_j}(t) \le \epsilon \nu \kappa_2 M$  for $t\in [U \wedge \epsilon \kappa_2   M,V]$. Together with  \eqref{InitalCondition} and \eqref{Event}, this implies that, on $G_M$,
\begin{equation} \label{changedfromalign}
\left[at/J - \epsilon b  T \right]  -\epsilon \nu \kappa_2   M 
\,\leq  \, Q_{\mA_j}(t) \, = \, Q_{\mA_j}(0) + A_{\mA_j}(t)  - D_{\mA_j}(t)
\, \leq \, \epsilon M /\nu + \left[at/J + \epsilon  b T \right]   \ .
\end{equation}
for  $t \in [U, V]$ and $j=1,...,J$.

Comparison of the upper and lower bounds in (\ref{eq:QABounds_a})--(\ref{eq:QABounds_b}) with those of (\ref{changedfromalign}) (after multiplying the latter by $\nu$) provides upper and lower bounds on $V$:   
For any $t > U$ at which the bound in \eqref{eq:QABounds_a} is at most the left-hand bound in (\ref{changedfromalign}), $V\le t$ must hold.
Similarly, at the last $t > U$ at which         the bound in \eqref{eq:QABounds_b}  is at least the right-hand side of (\ref{changedfromalign}), $V \geq t$.  Combining these upper and lower bounds for $V$, it is easy to check that
\begin{equation}\label{theinterval}  
V \in
\left[
	\frac{JM}{J+\nu} \cdot \frac{1}{1-a+ {a\nu }/{J}}
	- \epsilon \kappa_3 M
	, \,\,
			\frac{JM}{J+\nu} \cdot \frac{1}{1-a+ {a\nu }/{J}}
	+ \epsilon \kappa_3 M 
	\right]
\end{equation}
for large $M$, where $\kappa_3> 0$ does not depend on  
$\epsilon$ or $M$.  This implies \eqref{eq:Lem8A}.  Since $T > M/(1-a)$, $V < T$ follows immediately from the upper bound in (\ref{theinterval}), for sufficiently small $\epsilon >0$.

The inequality \eqref{eq:Lem8B} will now follow quickly.
By Lemma \ref{LemmaStopService}, all service  at any
Queue $\mA_j$, $j=1,\ldots,J$, by time $V$ must occur by time  $\epsilon \kappa_2 M$, which implies at most
$\epsilon \nu \kappa_2 M$ jobs can be served at any such queue by time $V$.   On the other hand, denoting by $t_*$ the lower bound in (\ref{theinterval}) and applying \eqref{Event}, at least  $at_*/J - \epsilon b T$ jobs arrive at each such queue by time $V$.  It follows from this and the definition of $V$ that
\begin{align*}
Q^{\Sigma}_{\mA}(V)
\geq
\big(J+\nu\big)
\min_{j=1,...,J} Q_{\mA_j}(V)
&
\geq
\big(J+\nu\big)
\Big[
at_*/J - \epsilon b T  -\epsilon \nu \kappa_2 M
\Big]
\geq
\frac{aM}{1 -a + a\nu /J}
-
\epsilon \kappa_4   M \,,
\end{align*}
where $\kappa_4 > 0$ is defined in terms of $\kappa_2$ and $\kappa_3$. 
%
\end{proof}

\smallskip
\smallskip

Employing Lemmas \ref{LemmaStopService}--\ref{Lemma8}, we now demonstrate Proposition \ref{MainProp}.

\smallskip

\begin{proof}[{Proof of Proposition \ref{MainProp}.}]

We need to show that, at time $V$, the conditions (\ref{Init1})-(\ref{Init3}) and (\ref{Init4}) hold.

Since $M':=\gamma \, M$, with 
$1 < \gamma < a/(1-a+{a \nu}/{J})$,  \eqref{Init1} follows immediately from (\ref{eq:Lem8B}) of
Lemma \ref{Lemma8},  
for $\epsilon > 0$ chosen sufficiently small.

To show \eqref{Init2} and (\ref{Init3}), we recall that
$ b = (\gamma - 1)(1-a)(1-r_{\vecrho})/6 \nu  J$.
Hence, by Lemma \ref{LemmaStillEmpty}, 
\begin{equation}\label{1st.ineq.epsilon.espisol_prime}
	Q_{\mB}^{\Sigma}(V) \le  \nu^3  +  \epsilon bJ T  \leq \epsilon \gamma M
\end{equation}
for large $M$,
and so \eqref{Init2} holds.
By the remark below Lemma \ref{Lemma8},  $|Q_{\mA_0} (V) - \nu \max_{j=1,...,J} Q_{\mA_j}(V)|\le \nu$.  It follows from this, Lemma \ref{Lemon9}, and the definition of $b$ that 
$$ \max_{j=1,...,J} \left| Q_{\mA_0}(V) - \nu Q_{\mA_j} (V) \right| \le
\epsilon M +  3 \epsilon b \nu  T + \nu
\leq \epsilon \gamma M$$
for large $M$, and so (\ref{Init3}) holds.

We still need to show \eqref{Init4}.
By Lemma  \ref{LemmaStopService}, Queues $\mA_j$, $j=1,\ldots, J$, are not served over times $t\in (\epsilon\kappa_2  M, V]$ and so, on $G_M$,
\begin{equation*}
	Q^{\Sigma}_\mA (t)
	\geq
	\sum_{j=1}^J Q_{\mA_j}(t)
	\geq
	\sum_{j=1}^J
		\left(
		A_{\mA_j}(t)
		-
		A_{\mA_j}(\epsilon \kappa_2   M) \right) 	
		\geq
	 a (t-\epsilon \kappa_2   M) - \epsilon b  J T
\quad \text{ for } t\le V \, .
	\end{equation*}
Since Component $\mB $, serves at most $\nu$ jobs per unit of time, $Q_{\mB }^{\Sigma} (t) \geq M - t   \nu$ for all $t$.
	
These two inequalities imply that, for $t \le V$,
\begin{align}\label{eq:Qsums}
	Q^\Sigma_\mA (t) + Q^\Sigma_\mB (t)
	&\geq
	[ a (t-\epsilon \kappa_2   M) -  \epsilon b J  T]
	 \vee
	[M - t   \nu ]
	\\[.2em]
	&\geq
	\frac{aM}{a +\nu}
	- \frac{\epsilon \nu ( \kappa_2 aM + b J T)}{a+ \nu}\, ,\notag
\end{align}
where the right-hand side of \eqref{eq:Qsums} follows
by weighting the terms on either side of $\vee$ by 
$\nu / (a+\nu)$ and $a / (a +\nu)$, respectively. 
Condition \eqref{Init4} follows by choosing $\epsilon >0$ sufficently small.
\end{proof}


\section{Stability Results.}\label{sec:stable}
Our main result in this section is Theorem \ref{thrm:p-MW}, which states that the queueing network process associated with the 
weighted MaxWeight optimization
(\ref{eq1wtmaxwt}) is positive recurrent under $\vecrho \in \mC$.

The proof of Theorem \ref{thrm:p-MW} employs the standard tool of fluid models. In Subsection \ref{sec:Fluid_Model}, we present a set of fluid model equations that are deterministic analogs of the queueing network equations corresponding to the switched network. 
Using the Lyapunov function
\begin{equation} \label{defofh}
h(t): 
=
\max_{{\vecsigma} \in {\mS}}\; \sum_{j=1}^J Q_{j}(t) 
\left( 
\frac{\sigma_{j}}{\rho_{j}}
- 1
\right)
=
\sum_{j\in \mathcal J} 
	Q_j (t)
		\left(\frac{D'_j(t)}{\lambda_j} -1 \right),
\end{equation}
 we show that, for $\bm {\rho} \in \mC$, the fluid model is stable.  We then follow standard reasoning to show that the positive recurrence of the queueing network process follows from the stability of the fluid model.  Details of the reasoning are given in the Appendix.

Since the proof of maximum stability is easier for the queueing network process associated with the  MaxWeight optimization
in the setting of tandem networks with constant mean work rates, we first prove the corresponding Theorem \ref{tandemtheorem} in Subsection \ref{sec.re-entrant.lines}.  We also state the analogous result, Theorem \ref{purebranching}, for networks with only branching, which is proved in the Appendix. In Subsection \ref{subsectionweightmaxweight}, we prove Theorem  \ref{thrm:p-MW}.

\subsection{Fluid Model.}\label{sec:Fluid_Model}


We employ here the \emph{fluid model equations}
\begin{subequations}\label{eq:FN}
\begin{align}
	\bar{Q}_j(t) 
&= 
\bar{Q}_j(0) 
+ \bar{A}_j(t) 
-\bar{D}_j(t)\, , \label{eq:FN1}
\\[.1em]
\bar{A}_j(t) 
&=
a_j t
+
\sum_{j'\in \mathcal J}
	\bar{D}_{j'}(t)P_{j'j} , \label{eq:FN2}
\\
{\bar{D}_j(t) } 
&= p_j{ \bar{\Pi}_j(t)}   \, \label{eq:FN3}
\end{align}
\end{subequations}
for $t\in \mathbb R_+$, $j,j' \in \mathcal J$, where $a_j$ and $p_j$ are as in Section \ref{sec:model} .  We denote the corresponding \emph{fluid model solutions} by
$({\bf \bar{Q}}(t),{\bf \bar{A}}(t),{\bf \bar{D}}(t),{\bf \bar{\Pi}}(t) : t\geq 0)$, where the components are the vectors corresponding to the terms in (\ref{eq:FN}).

As in Section  \ref{sec:model}, one can interpret $\bar{Q}_j(t)$ as the queue size of Queue $j$, $\bar{A}_j(t)$ as the cumulative number of arrivals, $\bar{D}_j(t)$ as the cumulative number of departures, and $\bar{\Pi}_j(t)$ as the cumulative amount of service scheduled at Queue $j$ by time $t$. One requires that $\bar{Q}_j(t)$ be non-negative, that $\bar{A}_j(t)$, $\bar{D}_j(t)$ and $\bar{\Pi}_j(t)$ be non-decreasing and non-negative, and
that $\bar{\Pi}_j(t)$ be Lipschitz continuous; the last implies that $\bar{D}_j(t)$, $\bar{A}_j(t)$ and $\bar{Q}_j(t)$ are also Lipschitz continuous. Lipschitz continuity implies these functions are almost everywhere differentiable.

We further assume ${\bf \bar{\Pi}}(t)$ satisfies the weighted MaxWeight optimization

\begin{equation}\label{argmax}
\bar{\bm{\Pi}}'(t) \in  \argmax_{{\vecsigma} \in {\mS}}\; \sum_{j\in\mJ} \bar{Q}_j(t)\frac{\sigma_j}{\rho_j}   
\end{equation}
corresponding to (\ref{eq1wtmaxwt}).

%

We say that a set of fluid model equations is \emph{stable} if there exists a time $t_0 >0$ such that, for all initial states $\bar{\bm Q}(0)$ with $| \bar{\bm Q}(0) | \leq 1$, all solutions of the fluid model equations satisfy
\[
\bar{\bm Q} (t) =0,\qquad \forall \, t \geq t_0 \,.
\]
It has been well-known since Dai \cite{Da95} (see also \cite{Br08}) that, under appropriate conditions, the stability of a fluid model will imply positive recurrence of the associated queueing network. In our case, we employ the following proposition, whose proof is given in Section \ref{Sec:PosRec} of the Appendix. 

\begin{proposition}
\label{propFMQN}
Suppose that the weighted MaxWeight fluid model (\ref{eq:FN})--(\ref{argmax}) is stable.  Then the corresponding weighted MaxWeight queueing network given in (\ref{eq1wtmaxwt})--(\ref{eq:PI}) is positive recurrent.
\end{proposition} 
 
In this subsection, we have used the bar superscript to distinguish the fluid model from its associated queueing model; in the remainder of Section \ref{sec:stable}, we will remove this superscript from the notation and work exclusively with the fluid model.

\subsection{Stability in Tandem Switched Networks.}\label{sec.re-entrant.lines}
A \emph{tandem network} is a queueing network with a single route, with each queue having a single class (see Figure \ref{Fig4}).
Tandem networks  are often used as models for  manufacturing systems.
Theorem \ref{tandemtheorem} shows that, when the mean work per job required at each queue (i.e., mean job size) is the same at all queues, MaxWeight is maximally stable.

\begin{figure}[ht!]
		\centering

\begin{tikzpicture}

	\tikzset{
		pic shift/.store in=\shiftcoord,
		pic shift={(0,0)},
		queue/.pic = {
			\begin{scope}[shift={\shiftcoord}]
				\coordinate (_arr) at (-0.25,0);
				\coordinate (_in) at (0.5,0);
				\coordinate (_dep) at (1,0);
				\draw (0,0.25) -- (1,0.25) -- (1, -0.25) -- (0,-0.25);
			 \end{scope}
		}
	}
	
		\path pic (A1) at (0.25,0) {queue} (A1_in) node {};
		\path pic (A2) at (2.5,0) {queue} (A2_in) node {};
		\path pic (AbJ) at (5.75,0) {queue} (AbJ_in) node {};
		\path pic (AJ) at (8,0) {queue} (AJ_in) node {};
		\draw [->] (A1_dep.east) -- (A2_arr.west) node [above, midway]{$\phantom{\nu}$};
		\draw [dotted, ->] (A2_dep.east)  -- ++(0.75,0) node [right] {$\cdots$} ++(0.75,0) -- (AbJ_arr.west);
		\draw [->] (AbJ_dep.east) -- (AJ_arr.west) node [above, midway]{$\phantom{\nu}$};
		\draw [->] (AJ_dep.east)  -- ++(1,0) node [above, midway]{$\phantom{\nu}$} -- ++(0.5,0) ;
		\draw [<-] (A1_arr.west) -- ++(-1,0) node [left]{$a$};
		
		\end{tikzpicture}

		\caption{A typical tandem network with external arrival rate $a$. 	\label{Fig4}}
\end{figure}
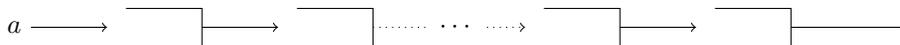

\smallskip

\begin{proof}[{Proof of Theorem \ref{tandemtheorem}.}]
The topology of the network and the unit size of the jobs imply, by \eqref{rel.rho.lambda.p}, that $\rho_{j}=a$ for all $1 \leq j \leq J$;
the condition $a < \nu$ is therefore equivalent to assuming $\bm {\rho} \in \mC$.
To analyze the fluid model defined in (\ref{eq:FN})--(\ref{argmax}), 
we employ the Lyapunov function
\begin{equation} \label{defofhfortandem}
h(t)
=
\sum_{j=1}^J 
Q_{ j}(t) 
\bigg(\frac{D'_{ j}(t)}{a} - 1 \bigg) \ ,
\end{equation}
which is the restriction of $h(t)$ in (\ref{defofh}) under
$\rho_{ j} = a$, for all $1 \leq j \leq J$.

Since $\rho \in \mathcal C$,  $h(t)$ is non-negative and is zero only when $\bm Q(t)=0$.  It follows from the definition of $h(t)$ in (\ref{defofh}) and the Lipschitz continuity of $\bm Q(t)$ that
 $h(t)$ is also Lipschitz continuous, and is thus almost everywhere differentiable.  

Most of the proof of the theorem consists of obtaining the upper bound (\ref{h'bound}) on the left derivative of $h'(t)$ (which equals the derivative almost everywhere).
It follows from (\ref{h'bound}) that $h(t) = 0$ for $t\ge Ch(0)$ for some constant $C$.  Since $h(0) \leq C' \sum_j Q_{ j}(0)$ for appropriate $C'$, this implies that the fluid model is stable. 
It then follows from Proposition \ref{propFMQN} that the queueing network is positive recurrent.

To bound the left derivative of $h(t)$, we first observe that, for $\delta < 0$, 
\begin{align*} 
\frac{h(t+\delta) - h(t)}{\delta}  
&
\leq 
\frac{1}{\delta}
\left[ 
	\sum_{j=1}^J 
		Q_{ j}(t+\delta) 
		\bigg(\frac{D'_{ j}(t)}{a} - 1 \bigg) 
 - 
 \sum_{j=1}^J  
 		Q_{ j}(t) 
 		\bigg(\frac{D'_{ j}(t)}{a} - 1 \bigg) 
\right]
\\
&
=
\sum_{j=1}^J 
	\frac{ Q_{ j}(t+\delta)  - Q_{ j}(t) }{\delta}
\bigg(\frac{D'_{ j}(t)}{a} - 1 \bigg) .
\end{align*}
The above inequality follows from the optimality for MaxWeight of 
$({D}'_{ j}(t+\delta): 1 \leq j \leq J)$ for queue sizes  
$(Q_{ j}(t+\delta) : 1 \leq j \leq J)$, and the suboptimality of $({D}'_{ j}(t): 1 \leq j \leq J)$.
Taking limits implies that
\begin{align*}
h'(t) 
= 
\lim_{\delta \nearrow 0} 
	 \frac{h(t+\delta) - h(t)}{\delta}  
&
\leq 
\lim_{\delta  \nearrow 0} 
\sum_{j=1}^J
	\frac{ Q_{ j}(t+\delta)  - Q_{ j}(t) }{\delta}
\bigg(\frac{D'_{ j}(t)}{a} - 1 \bigg) 
\\
&
=
\sum_{j=1}^J Q_{ j}'(t) \bigg(\frac{D'_{ j}(t)}{a} - 1 \bigg)  .
\end{align*}

 Expanding the last term in the above display implies
\begin{align}
{h}' (t)
&
\leq 
\sum_{j=1}^J Q_{ j}'(t) \bigg(\frac{D'_{ j}(t)}{a} - 1 \bigg) 
\notag 
 = 
 \sum_{j=1}^J Q_{ j}'(t) \frac{D_{ j}'(t)}{a}  
- (a - D'_{ J}(t))
\notag 
\\[.15em]
&
=  
\sum_{j=1}^{J+1} \Big( D'_{ {j-1}}(t)-D'_{ j}(t)\Big)
\frac{
 D_{ j}'(t)
}{a}
\notag 
= 
- \frac{1}{a} \sum_{j=1}^{J+1}
\left[
	\frac{1}{2} D'_{ j}(t)^2
	-
	D'_{ j}(t) D'_{ {j-1}}(t)
	+
	\frac{1}{2} D'_{ {j-1}}(t)^2
\right]	
\notag 
\\
&=
-\frac{1}{2 a} \sum_{j=1}^{J+1} \Big( D'_{ {j-1}}(t)-D'_{ j}(t)\Big)^2 \, ,
\label{eq:h_bound}
\end{align}
where we are setting $D'_{ 0}(t)=a$ and $D'_{ {J+1}}(t)=a$. 
The first equality in (\ref{eq:h_bound}) holds since $\sum_j Q'_{ j}(t)$ telescopes, and the third equality employs the identity
\[
\sum_{j=1}^{J+1} D'_{ j}(t)^2 
=
\sum_{j=1}^{J+1}
\left[
	\frac{1}{2}D'_{ j}(t)^2
	+
	\frac{1}{2}D'_{ {j-1}}(t)^2
\right]\ .
\]

To bound the last term in (\ref{eq:h_bound}), we employ $\bm \rho = (a,...,a) \in \mC$ and $\bm D'(t)\in\partial\mC$, 
which state that the traffic intensity lies in the interior of the stability region,
and that the departure rate vector lies on its boundary whenever queue sizes are non-zero.
In particular, there exists a constant $\epsilon>0$, not depending on time, such that $|a - D'_{ {j_0}}(t) | \geq \epsilon $ 
for some $j_0 \in \{1, \ldots,J\}$, whenever $\bm Q(t) \neq 0$.
It follows that
\begin{equation}\label{eq:epsi_bound}
\sum_{j=1}^{J+1}
\Big( D'_{ {j-1}}(t)-D'_{ j}(t) \Big)^2
\geq 
\sum_{j=1}^{j_0}
\Big( D'_{ {j-1}}(t)-D'_{ j}(t) \Big)^2
\geq 
\frac{
\epsilon^2
}{
j_0^2
}
\geq 
\frac{\epsilon^2 }{J^2}\,,
\end{equation}
since at least one of the differences must be at least $\epsilon / j_0$.

When $\bm Q(t) \neq 0$, \eqref{eq:h_bound} and \eqref{eq:epsi_bound} imply that 
\begin{equation} \label{h'bound}
h'(t) \leq - \frac{\epsilon^2}{2a J^2}\,.
\end{equation}
It follows that $h(t)=0$  for all $t>t_0:=2a J^2 h(0)/\epsilon^2$, which implies that $\bm Q(t)=0$ for all $t\ge t_0$, as desired. 
\end{proof}

\medskip

Generalizations of Theorem \ref{tandemtheorem} hold for certain other networks, such as switched networks consisting of parallel sequences of queues in tandem that do not interact, if one assumes that the mean required work is constant over queues along individual sequences.  Maximum stability of the MaxWeight policy for such switched networks is a simple application of Theorem \ref{thrm:p-MW}, since scaling the required work along different sequences does not affect the associated queueing network.

More significant examples are  \emph{pure branching switched networks}, i.e., open switched networks with a single arrival stream at each queue.  An example is given in Figure \ref{Fig5}. 
As before, the mean job sizes at different queues are assumed to be equal.  The stability result for pure branching networks is stated in Theorem \ref{purebranching} and its proof is given in the Appendix.  Interestingly, the result cannot be derived from Theorem \ref{thrm:p-MW}.

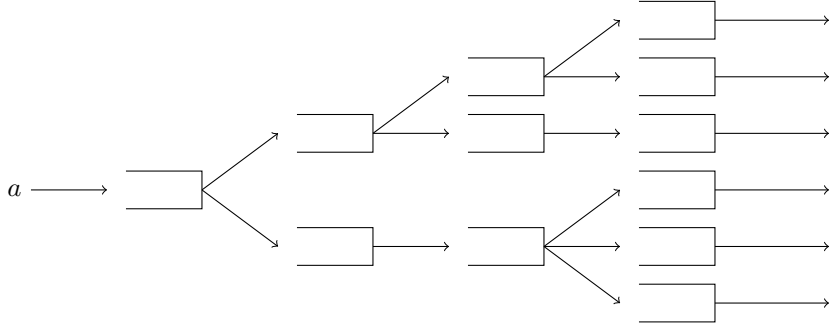
\begin{figure}[h]
	\centering

\begin{tikzpicture}

\tikzset{
    pic shift/.store in=\shiftcoord,
    pic shift={(0,0)},
    queue/.pic = {
        \begin{scope}[shift={\shiftcoord}]
            \coordinate (_arr) at (-0.25,0);
            \coordinate (_in) at (0.5,0);
            \coordinate (_dep) at (1,0);
            \draw (0,0.25) -- (1,0.25) -- (1, -0.25) -- (0,-0.25);
         \end{scope}
    }
}

    at (-0.5,1.15) {\raisebox{10em}{}};
\path pic (A1) at (0.25,0.75) {queue} (A1_in) node {};
\path pic (A11) at (2.5,0) {queue} (A11_in) node {};
\path pic (A12) at (2.5,1.5) {queue} (A12_in) node {};
\path pic (A111) at (4.75,0) {queue} (A111_in) node {};
\path pic (A121) at (4.75,1.5) {queue} (A121_in) node {};
\path pic (A122) at (4.75,2.25) {queue} (A122_in) node {};
\path pic (A1111) at (7,-0.75) {queue} (A1111_in) node {};
\path pic (A1112) at (7,0) {queue} (A1112_in) node {};
\path pic (A1113) at (7,0.75) {queue} (A1113_in) node {};
\path pic (A1211) at (7,1.5) {queue} (A1211_in) node {};
\path pic (A1221) at (7,2.25) {queue} (A1221_in) node {};
\path pic (A1222) at (7,3) {queue} (A1222_in) node {};
\draw [->] (A1_dep.east) -- (A11_arr.west);
\draw [->] (A1_dep.east) -- (A12_arr.west);
\draw [->] (A11_dep.east) -- (A111_arr.west);
\draw [->] (A12_dep.east) -- (A121_arr.west);
\draw [->] (A12_dep.east) -- (A122_arr.west);
\draw [->] (A111_dep.east) -- (A1111_arr.west);
\draw [->] (A111_dep.east) -- (A1112_arr.west);
\draw [->] (A111_dep.east) -- (A1113_arr.west);
\draw [->] (A121_dep.east) -- (A1211_arr.west);
\draw [->] (A122_dep.east) -- (A1221_arr.west);
\draw [->] (A122_dep.east) -- (A1222_arr.west);
\draw [->] (A1111_dep.east)  -- ++(1.5,0);
\draw [->] (A1112_dep.east)  -- ++(1.5,0);
\draw [->] (A1113_dep.east)  -- ++(1.5,0);
\draw [->] (A1211_dep.east)  -- ++(1.5,0);
\draw [->] (A1221_dep.east)  -- ++(1.5,0);
\draw [->] (A1222_dep.east)  -- ++(1.5,0);

\draw [<-] (A1_arr.west) -- ++(-1,0) node [left]{$a$};

\end{tikzpicture}

	\caption{A pure branching network with external arrival rate $a$.\label{Fig5}}
\end{figure}

\begin{theorem}\label{purebranching}
For a pure branching switched network with
$
\bm \rho  \in {\mathcal C}$ and 
the same mean work per job required at each queue,
 the associated queueing network process $\mQ$ is positive recurrent under the MaxWeight policy.
\end{theorem}

\subsection{Stability of Weighted MaxWeight.}
\label{subsectionweightmaxweight}
We now consider the stability of the weighted MaxWeight policy in \eqref{eq1wtmaxwt}.
Since the scheduling policy depends on $\bm \rho$, weighted MaxWeight is less applicable in practice than is MaxWeight.  Theorem \ref{thrm:p-MW}  shows, however, that the queueing network process associated with the weighted MaxWeight policy is positive recurrent under $\vecrho \in \mC$, without any restriction on the network topology or mean workload. 
The main tool in the proof of Theorem \ref{thrm:p-MW} will be the Lyapunov function $h(t)$ defined in (\ref{defofh}), which  indicates why the weighting by $\bm \rho$ is needed mathematically to 
stabilize MaxWeight in multihop networks.
%



Much of the work needed for Theorem \ref{thrm:p-MW} is done in the following lemma.  We write $h(t) = f(t) - g(t)$, and analyze the two parts separately.
\begin{lemma}\label{lem:technical}
Define the functions $f(t)$, $g(t)$ and $h(t)$ by
\[
f(t)
=
\sum_{j\in \mathcal J} 
	Q_j (t)
		\frac{D'_j(t)}{\lambda_j}\,,
\qquad
g(t) =\sum_{j\in \mathcal J} 
	Q_j (t) \,,
\qquad 
\text{and}
\qquad 
h(t) 
= 
f(t)-g(t)\,.
\]
Then
	\begin{equation*}
	h'(t)
	\leq 
	-
	\frac{1}{2}
	\sum_{j,j'\in \mathcal J} 
		\lambda_j P_{jj'} 
		\left( \frac{D'_j(t)}{\lambda_j}-  \frac{D'_{j'}(t)}{\lambda_{j'}} \right)^2
	-
	\frac{1}{2}
	\sum_{j\in \mathcal J } 
		a_j
		\left( 
			\frac{D'_j(t)}{\lambda_j} - 1 
		\right)^2 
	-
	\frac{1}{2}
	\sum_{j\in \mathcal J}
		\lambda_j
		P_{j\omega} 
		\left( 
			\frac{D'_{j}(t)}{\lambda_j}
			-
			1
		\right)^2\, . 
\end{equation*}
\end{lemma}

\begin{proof}
Using reasoning similar to that in the proof of Theorem \ref{tandemtheorem}, we bound the left
derivative of $f(t)$ from above by first observing that, for $\delta < 0$,
\begin{align*} 
\frac{f(t+\delta) - f(t)}{\delta}  
 \leq 
\sum_{j\in \mathcal J} 
	\frac{ Q_{j}(t+\delta)  - Q_{j}(t) }{\delta} \cdot
\frac{{D}'_{j}(t)}{\lambda_j} \, ,
\end{align*}
which follows from the optimality for weighted MaxWeight of $({D}'_{j}(t+\delta): j\in\mathcal J)$ for queue sizes  $(Q_{j}(t+\delta) : \in\mathcal J)$ and the suboptimality of $({D}'_{j}(t): j\in\mathcal J)$. 
Taking limits implies that
\begin{align*}
f'(t) 
&
= 
\lim_{\delta \nearrow 0} 
	 \frac{f(t+\delta) - f(t)}{\delta}  
\leq  \,
\lim_{\delta \nearrow 0} 
\sum_{j\in \mathcal J}
	\frac{ Q_{j}(t+ \delta)  - Q_{j}(t) }{\delta}
\cdot \frac{{D}'_{j}(t) }{\lambda_j}
= \,
\sum_{j\in \mathcal J} Q_{j}'(t) \frac{D'_{j}(t)}{\lambda_j} \, .
\end{align*}

Rewriting this upper bound, we obtain

\begin{align}
f'(t) 
& \leq
\sum_{j \in \mathcal J}
 \Big( a_j + \sum_{j'\in \mathcal J} D'_{j'}(t) P_{j' j }  - D'_j(t)   \Big) \frac{D'_j(t)}{\lambda_j} 
 \notag
 \\[.1em]
& =
\sum_{j \in \mathcal J}
	a_j \frac{D'_j(t)}{\lambda_j}
-
\sum_{j\in\mathcal J} \lambda_j \left( \frac{D'_j(t)}{\lambda_j} \right)^2
+
\sum_{j, j'\in\mathcal J} \lambda_j {P_{jj'}} \left[ \frac{D'_{j}(t)}{\lambda_j} \cdot
\frac{D'_{j'}(t) }{\lambda_{j'}} \right]
\notag
\\[.05em]
& =
\sum_{j \in \mathcal J}
	a_j \frac{D'_j(t)}{\lambda_j}
-
\sum_{j\in\mathcal J} \lambda_j \left( \frac{D'_j(t)}{\lambda_j} \right)^2
\notag
\\[.25em]
 &\quad +
\sum_{j, j'\in\mathcal J} \lambda_j  {P_{jj'}}
\left[
-\frac{1}{2}\left(
	\frac{D'_{j}(t)}{\lambda_j}
	-
	\frac{D'_{j'}(t) }{\lambda_{j'}} 
\right)^2
+
\frac{1}{2}
\left(
\frac{D'_j(t)}{\lambda_j}
\right)^2
+
\frac{1}{2}
\left(
\frac{D'_{j'}(t)}{\lambda_{j'}}
\right)^2
\right] \ .
\label{f_Dash_1}
\end{align}
The inequality is obtained by applying the identities \eqref{eq:FN1} and \eqref{eq:FN2} to the right-hand side of the previous display, and the two equalities follow by rearranging the terms and applying the identity
 $x \, y = [x^2 + y^2-(x-y)^2]/2$  with $x={D'_{j}(t)}/{\lambda_{j}}$ and $y={D'_{j'}(t)}/{\lambda_{j'}}$ . 

 We sum the second term in brackets over $j'$ by introducing 
$P_{j \omega}:= 1 - \sum_{j' \in \mathcal J }  {P_{jj'}}$,
 then adding and subtracting the quantity $\frac{1}{2} \, \sum_{j\in \mathcal J} a_j \left({D'_j(t)}/{\lambda_j}\right)^2$, and relabeling the
 indexes $(j,j')$ as $(j',j)$ in the last term in the brackets to obtain 

\begin{align}
f'(t)
& 
\leq 
\sum_{j \in \mathcal J}
	a_j \frac{D'_j(t)}{\lambda_j}
-\frac{1}{2}
\sum_{j, j'\in\mathcal J} \lambda_j  {P_{jj'}}
\left(
	\frac{D'_{j}(t)}{\lambda_j}
	-
	\frac{D'_{j'}(t) }{\lambda_{j'}} 
\right)^2
- \frac{1}{2}
\sum_{j\in \mathcal J} 
	a_j 
	\left( 
		\frac{D'_j(t)}{\lambda_j}
	\right)^2
\notag 
\\
&
\quad 
-
\frac{1}{2} 
\sum_{j \in \mathcal J}
\lambda_j P_{j \omega} 
\left( 
\frac{D'_j(t)}{\lambda_j}
\right)^2
+
\frac{1}{2}
\sum_{j\in \mathcal J} 
\left\{
-\lambda_j
+
\sum_{j' \in \mathcal J} \lambda_{j'} P_{j'j}
+
a_j
\right\} 
\left( \frac{D'_j(t)}{\lambda_j}\right)^2
\notag
\\[.25em]
&
=
\sum_{j \in \mathcal J}
	a_j \frac{D'_j(t)}{\lambda_j}
-\frac{1}{2}
\sum_{j, j'\in\mathcal J} \lambda_j  {P_{jj'}}
\left(
	\frac{D'_{j}(t)}{\lambda_j}
	-
	\frac{D'_{j'}(t) }{\lambda_{j'}} 
\right)^2 
- \frac{1}{2}
\sum_{j\in \mathcal J} 
	a_j 
	\left( 
		\frac{D'_j(t)}{\lambda_j}
	\right)^2
\notag 
\\[.2em] 
& \quad 
-
\frac{1}{2} 
\sum_{j \in \mathcal J}
\lambda_j P_{j \omega} 
\left( 
\frac{D'_j(t)}{\lambda_j}
\right)^2 \ .
\label{f_Dash_2}
\end{align}
For the equality, we note that the term in braces is zero due the traffic equations \eqref{eq:Traffic}. 

On the other hand, since $\sum_{j'\in \mathcal J }  {P_{jj'}} = 1-P_{j \omega}$,
\begin{equation}\label{g_Dash}
	g'(t) = \sum_{j\in \mathcal J} Q'_j(t) 
=
\sum_{j\in \mathcal J}
 \Big( a_j + \sum_{j'\in \mathcal J} D'_{j'}(t) P_{j' j }   - D'_j(t)   \Big) \
=
\sum_{j\in \mathcal J} a_j 
- 
\sum_{j\in \mathcal J} D'_j(t) P_{j \omega} \ .
\end{equation}
Combining \eqref{f_Dash_2} with \eqref{g_Dash} implies that
\begin{align*}
	 h'(t) 
	&
	=
		f'(t) - g'(t)
	\\
	\leq 
	&
	\sum_{j \in \mathcal J}
	a_j \frac{D'_j(t)}{\lambda_j}
-\frac{1}{2}
\sum_{j, j'\in\mathcal J} \lambda_j  {P_{jj'}}
\left(
	\frac{D'_{j}(t)}{\lambda_j}
	-
	\frac{D'_{j'}(t) }{\lambda_{j'}} 
\right)^2
- 
\frac{1}{2}
\sum_{j\in \mathcal J} 
	a_j 
	\left( 
		\frac{D'_j(t)}{\lambda_j}
	\right)^2
\\
 & 
 -
\frac{1}{2} 
\sum_{j \in \mathcal J}
\lambda_j P_{j \omega} 
\left( 
\frac{D'_j(t)}{\lambda_j}
\right)^2
-
\sum_{j\in \mathcal J} a_j 
+
\sum_{j\in \mathcal J} D'_j(t) P_{j \omega} 
\\[.25em]
				=
				&
	-
	\frac{1}{2}
	\sum_{j,j'\in \mathcal J} 
		\lambda_j P_{jj'} 
		\left( \frac{D'_j(t)}{\lambda_j}-  \frac{D'_{j'}(t)}{\lambda_{j'}} \right)^2
	-
	\frac{1}{2}
	\sum_{j\in \mathcal J } 
		a_j
		\left( 
			\frac{D'_j(t)}{\lambda_j} - 1 
		\right)^2 
	-
	\frac{1}{2}
	\sum_{j\in \mathcal J}
		\lambda_j
		P_{j\omega} 
		\left( 
			\frac{D'_{j}(t)}{\lambda_j}
			-
			1
		\right)^2 \,;
\end{align*}
the equality uses  $\sum_{j \in \mathcal J }  a_j = \sum_{j \in \mathcal J }  \lambda_j P_{j\omega }$, which is obtained by summing \eqref{eq:Traffic} over $j$.
This gives the
 desired expression.
\end{proof}

\medskip

We now prove Theorem \ref{thrm:p-MW}. The proof is similar to that of Theorem \ref{tandemtheorem}. 

\medskip

\begin{proof}[{Proof of Theorem \ref{thrm:p-MW}.}]
We define $h(t)$ as in  (\ref{defofh}) and Lemma \ref{lem:technical}.  
Provided $\bm Q(t) \neq 0$, the weighted MaxWeight schedule $( \sigma_j: j\in \mathcal J )$ belongs to the boundary of the set $< \mS >$, whereas the vector $\bm \rho$ does not.  Therefore, $h(t) >0 $ whenever $\bm Q(t) \neq 0$ and,
for some fixed $\epsilon$ (depending only on $\bm \rho$ and $\mS$), there exists  $j^*$ such that
\begin{equation}\label{pareto.optimality}
\frac{D'_{j^*}(t)}{\lambda_{j^*}} - 1  > \epsilon\, .
\end{equation}

There exists a path $j_1, \ldots, j_L = j^*$, with $L \le |\mathcal J|$ and satisfying $a_{j_1} >0$ and $P_{j_{\ell}j_{\ell +1}}>0$ for $\ell = 1,\ldots L-1$.  Consequently, by
(\ref{pareto.optimality}),
%
\begin{equation} \label{pairofinequalities}
\text{either} \quad
	\left| 
	 \frac{D'_{j_{\ell^*-1}}(t)}{\lambda_{j_{\ell^*-1}}} 
	 -
	 \frac{D'_{j_{\ell^*}}(t)}{\lambda_{j_{\ell^*}}} 
	\right| 
	> 
	\frac{\epsilon}{|\mathcal J|}
	\,\,\,\text{ for some } \ell^* \in [2,j^*]\,,
	\quad\text{or}
	\quad 
	\left| 
	 \frac{D'_{j_1}(t)}{\lambda_{j_1}} 
	 -
	 1
	\right| 
	> 
	\frac{\epsilon}{|\mathcal J|}\, .
\end{equation}
Together with  Lemma \ref{lem:technical}, (\ref{pairofinequalities}) implies
%
%
 that,
for appropriate $c>0$ and all $t$ satisfying $\bm Q (t) \neq 0$, 
\begin{equation} \label{bound on h'}
h'(t) \leq  -c\frac{\epsilon^2}{ |\mathcal J|^2}\,.
\end{equation}

Recall that $h(t) \ge 0$, since $\rho \in \mathcal C$,   and $h'(t) \le 0$ by Lemma \ref{lem:technical}. Moreover, $h(0)$ is uniformly bounded for all
$\bm Q(0)$ satisfying $\sum_j Q_j(0)  = 1$.
For such $\bm Q(0)$,  (\ref{bound on h'}) implies
 a uniform bound  $t_0$ exists with  $h(t) = 0$ for all $t \geq t_0$.  Consequently, $\bm Q(t) = 0$ for all $t \geq t_0$.
By Proposition \ref{propFMQN}, the weighted MaxWeight queueing network is therefore positive recurrent.
\end{proof}

%

\section{Stability and Instability in Multiclass Switched Networks.} \label{sec:multiclass}

\begin{figure}

	\centering
%
\begin{tikzpicture}[scale=1]
\def\hsp{0.25} 
\def\vsp{0.5} 
\def\arrL{2} 
\def\stSep{4.75}

\tikzset{
    pic shift/.store in=\shiftcoord,
    pic shift={(0,0)},
    queue/.pic = {
        \begin{scope}[shift={\shiftcoord}]
            \coordinate (_arr) at (-0.25,0);
            \coordinate (_in) at (0.5,0);
            \coordinate (_dep) at (1,0);
            \draw (0,0.25) -- (1,0.25) -- (1, -0.25) -- (0,-0.25);
         \end{scope}
    }
}

	\node (A) [draw, thick, shape=rectangle, minimum width=3.5cm, minimum height=3.5cm, anchor=center] 
		at (0,0.5)  {\raisebox{6.5em}{$\mathcal{A}$}};
	\path pic (A0) at (-0.25,1) {queue} (A0_in) node {$\mathcal{A}_0$};
	\path pic [rotate=180] (A1) at (0.75,0) {queue} (A1_in) node {$\mathcal{A}_1$};
	
	\node (B) [draw, thick, shape=rectangle, minimum width=3.5cm, minimum height=3.5cm, anchor=center] 
		at (4.5,0.5) {\raisebox{-8em}{$\mathcal{B}$}};
	\path pic (B1) at (3.75,1) {queue} (B1_in) node {$\mathcal{B}_1$};
	\path pic [rotate=180] (B0) at (4.75,0) {queue} (B0_in) node {$\mathcal{B}_0$};

	\draw [<-] (A0_arr.west) -- ++(-1.5,0) node [left]{$1$};
	\draw [->] (A0_dep.east) -- (B1_arr.east);
	\draw [dotted, ->] (A1_dep.west) -| ++(-0.75,0) |- ([shift={(0.5,-.5)}]A1_arr.west) node [below, pos=0.75]{\tiny $K-1$ returns}  |- ([shift={(0,-.1)}]A1_arr.west);
	\draw [->] (A1_dep.west) ++(-0.75,0) -- ++(-1.25,0);
	\draw [<-] (B0_arr.east) -- ++(1.5,0) node [right]{$1$}; 
	\draw [dotted, ->] (B1_dep.west) -| ++(+0.75,0) |- ([shift={(-0.5,.5)}]B1_arr.west) node [above, pos=0.75]{\tiny $K-1$ returns}  |- ([shift={(0,.1)}]B1_arr.west);
	\draw [->] (B1_dep.east) ++(0.75,0) -- ++(1.25,0); 
	\draw [->] (B0_dep.west) -- (A1_arr.east);
		
\end{tikzpicture}
%
	\quad
%
\begin{tikzpicture}[scale=0.25]


	\draw[fill=gray!20] (0,0) -- 
		(3,0) node[below]{\tiny $1/\epsilon$} -- 
		(0,9) node[left]{\tiny $1/\epsilon^2$} -- 
		cycle;

	\draw[->, thick] (0,0)--(10,0) node[right]{$\sigma_{\mA_0}$};
	\draw[->, thick] (0,0)--(0,10) node[above]{$\sigma_{\mA_1}$};
	
	\node[circle, draw=white, fill=black,line width=0.2mm, inner sep=0pt, minimum size=2.5pt
	] at (0.29,7.9) {};
	\node[circle, draw=white, fill=white, inner sep=0pt, minimum size=0.5pt
	] at (0.29,7.9) (A) {};
	\draw [<-](0.7,7.9) -- (2,7.9) -- (4,6.9) -- (4.7,6.9) node[right]{\tiny $(\rho_{\mA_0},\rho_{\mA_1})=(1,K)$};

		
\end{tikzpicture}
%
	\caption{A multiclass switched queueing network. 
	On the right is the projection on the space $(\sigma_{\mA_0},\sigma_{\mA_1})\in\bR^2$ of the set $<\mS>$ given in (\ref{formulticlasscounterexample}).
	\label{Fig6}}
\end{figure}
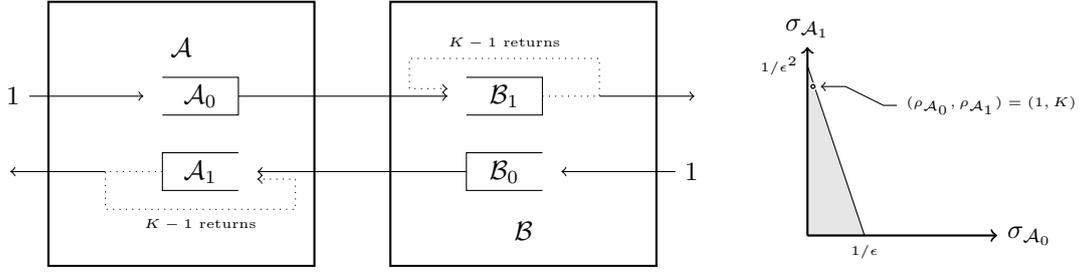

The networks that we have considered so far have been single class.
%
 In this section, 
we discuss results on the stability of MaxWeight for multiclass switched queueing networks.

%
%



A multiclass queueing network is a network that permits more than one class of job at a queue; we label the set of classes belonging to a queue $j$ by $\mCl(j)$ and by  ${\mathcal K} $ the set of all classes of jobs.   
Different classes at a queue can differ in the service times of their jobs and in the routes the jobs might take after service is completed at the queue.

Employing  definitions and notation that are analogous to those in Section \ref{sec:model}, we assume that the number of external arrivals of class $k$ at each time are independent identically distributed random variables with mean $a_{k}$. A class $k$ job has a size that is geometrically distributed on $\bZ_+$ with parameter $p_{k}$. (After receiving one unit of service at queue $j$, a class $k$ job, $k\in \mCl(j)$, has probability of departure $p_{k}$.)
 When a class $k$ job is served, the job is routed to class $k'$ with probability $
P_{k k'}$ and otherwise leaves the network. 
  The total arrival rate $\lambda_k$ of jobs at class $k$ is given by the solution to the traffic equations
\begin{equation*}
	\lambda_k
	= 
	a_k
	+
	\sum_{k'\in \mathcal K}
	\lambda_{k'} 
	P_{k' k}\, , \qquad k \in \mathcal K\,.
\end{equation*}

The traffic intensity is $\bm \rho = (\rho_j : j \in \mathcal J)$, where  $\rho_j = \sum_{k\in \mCl(j)} \tilde{\rho}_{k}$ and $\tilde{\rho}_{k} =\lambda_k/ p_k$.
The scheduling set $\mS \subset \bZ^{|\mJ|}_+$  is defined as in Section \ref{sec:intro}, with the subcritical region $\mC$ being the interior of 
$<\mS>$ and the associated switched queueing network process $\vecQ(t)$ being subcritical if $\bm \rho \in \mC$. The MaxWeight policy is again given in display (\ref{eq1}), where $Q_j = \sum_{k\in \mCl(j)}\tilde{Q}_{k}$ is the total number of jobs at queue $j$, with $\tilde{Q}_{k}$ being the number of jobs in class $k$, and  $\sigma_j = \sum_{k\in \mCl(j)}\tilde{\sigma}_{k}$ is the work devoted to queue $j$, with $\tilde{\sigma}_{k}$ being the work devoted to class $k$.


Unlike single class switched queueing networks, the order of service of jobs in a queue can influence the stability of a multiclass switched network.
In Subsection \ref{sec:multi-instab},  we will assume the policy is first-in, first-out (FIFO), with jobs being served at a queue in the order of their arrival at their current class.  We will furthermore assume that the switched queueing network processes are of Kelly type, i.e., the mean service times satisfy $m_k = m_{k'}$ for all classes $k,k' \in \mCl(j)$, $j\in \mJ$. We present an example (without proof) of such a subcritical switched queueing network for which the associated queueing network process is transient.  This behavior is distinct from that for the ProportionalScheduler policy in Section \ref{sec:Conc} and for queueing network processes without switching, where subcritical FIFO queueing networks of Kelly type must be positive recurrent.  The example is depicted in Figure \ref{Fig6}.

In Subsection \ref{stableMC}, we will consider the policy that discriminates between job classes within a queue by first serving the class with the greatest weighted number of jobs.   Reasoning similar to that in the proof of Theorem \ref{thrm:p-MW} shows that the subcritical switched queueing network process for this non-FIFO policy is always positive recurrent.

\subsection{A Family of Unstable Multiclass Switched Queueing Networks of Kelly Type.}\label{sec:multi-instab}

For the multiclass switched queueing network in Figure \ref{Fig6}, jobs arrive at rate $1$ at the single class queues $\mA_0$ and $\mB_0$, and require $1$ unit of service at each of these queues and $1$ unit of service at each of the $K$ classes of the queues 
$\mA_1$ and $\mB_1$.  So, the corresponding switched queueing network process ${\bf Q}(t)$ is of Kelly type; 
moreover, $\tilde{\rho}_k =1$ for all $k$.
We choose $\epsilon >0$ so that $K= (1-2\epsilon)/\epsilon^2$ and denote the $K$ classes at $\mA_1$ and $\mB_1$ by $\mA_1^{(1)}, \ldots, \mA_1^{(K)}$ and $\mB_1^{(1)},\ldots,\mB_1^{(K)}$, respectively.

The set $\mS$ of feasible service options is defined by
\begin{align}
\label{formulticlasscounterexample}
{\epsilon \sigma_{\mathcal{A}_0}}{}
+
\epsilon^2 \sigma_{\mA_1}
\leq 1
\qquad  \text{and} \qquad
{\epsilon \sigma_{\mathcal{B}_0}}
+
\epsilon^2 
	 \sigma_{\mB_1}
\leq 1\, .
\end{align}
It is easy to check that $\bm \rho \in \mC$ for the above switched network.
On the other hand, for $\epsilon$ close to $0$, the process ${\bf Q}(t)$ will be transient.  We do not prove this assertion here, but the proof is similar to that for the family of networks in Theorem \ref{thrm:MW}.  

\smallskip

REMARKS.
The networks in Figure \ref{Fig2} and Figure \ref{Fig6} are both modeled after the Rybko--Stolyar network; the Lu--Kumar reentrant line could be employed in an analogous manner in each case.  
An example of a subcritical single class switched queueing network that is unstable is obtained by modifying Figure \ref{Fig6} by ``collapsing'' the $K$ classes in $\mA_1$ and $\mB_1$ into single classes having geometric service times with mean 
$m = (1-2\epsilon)/\epsilon^2$, with $\mS$ again being given by  
(\ref{formulticlasscounterexample}).  
The queueing network in Figure \ref{Fig6} and its collapsed analog will evolve similarly.  

\begin{figure}[ht!]
	\begin{minipage}{0.47\textwidth}
		\centering
		\includegraphics[width=\textwidth]{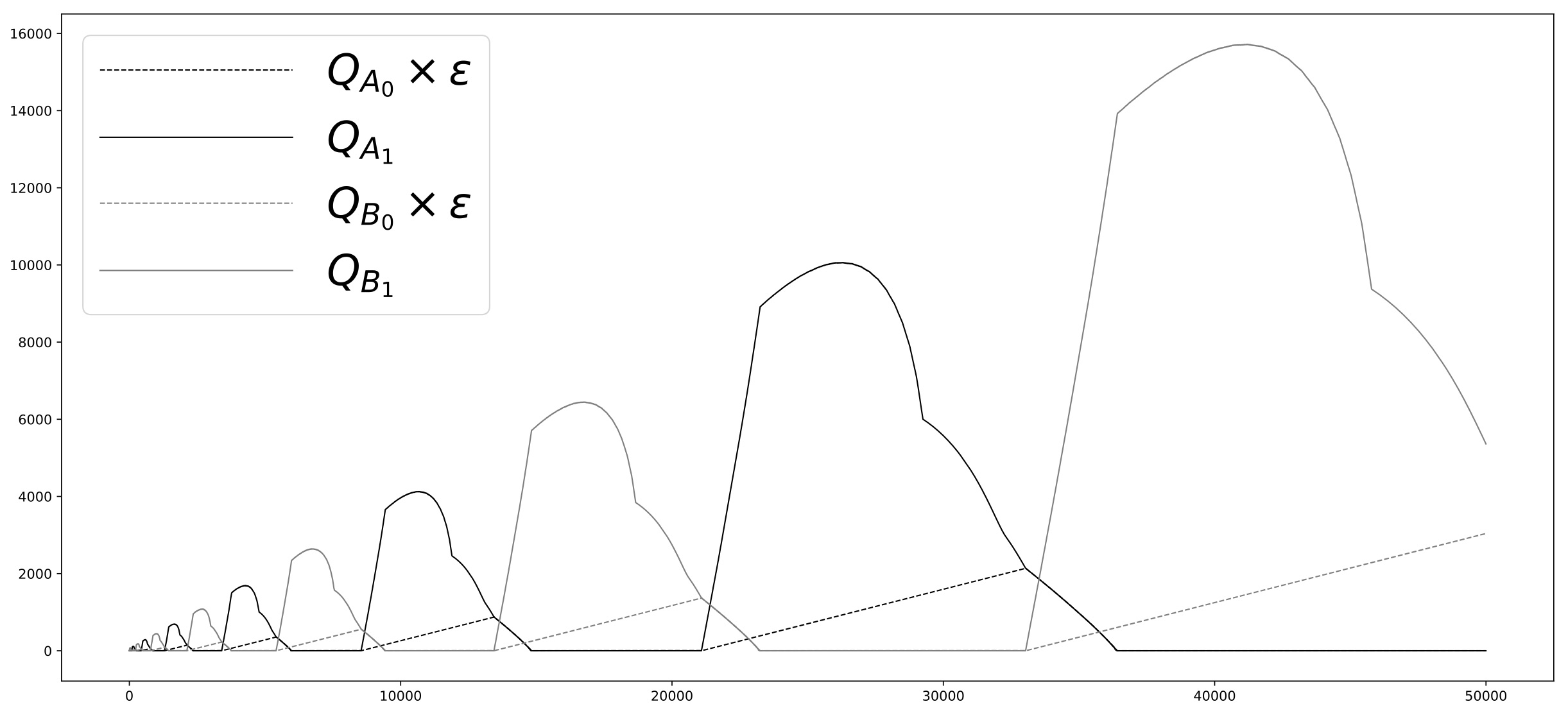}
		\caption{A simulation of the network depicted in Figure \ref{Fig6} for parameters
		$a = 1$, $\epsilon = 0.1791$ and $K = 20$. 
		\break
		The initial value is $Q_{\mA_0}(0)= 55$, with all other queues empty.
		\label{Fig7}}
	\end{minipage}
	\hfill
	\begin{minipage}{0.47\textwidth}
		\centering
		\includegraphics[width=\textwidth]{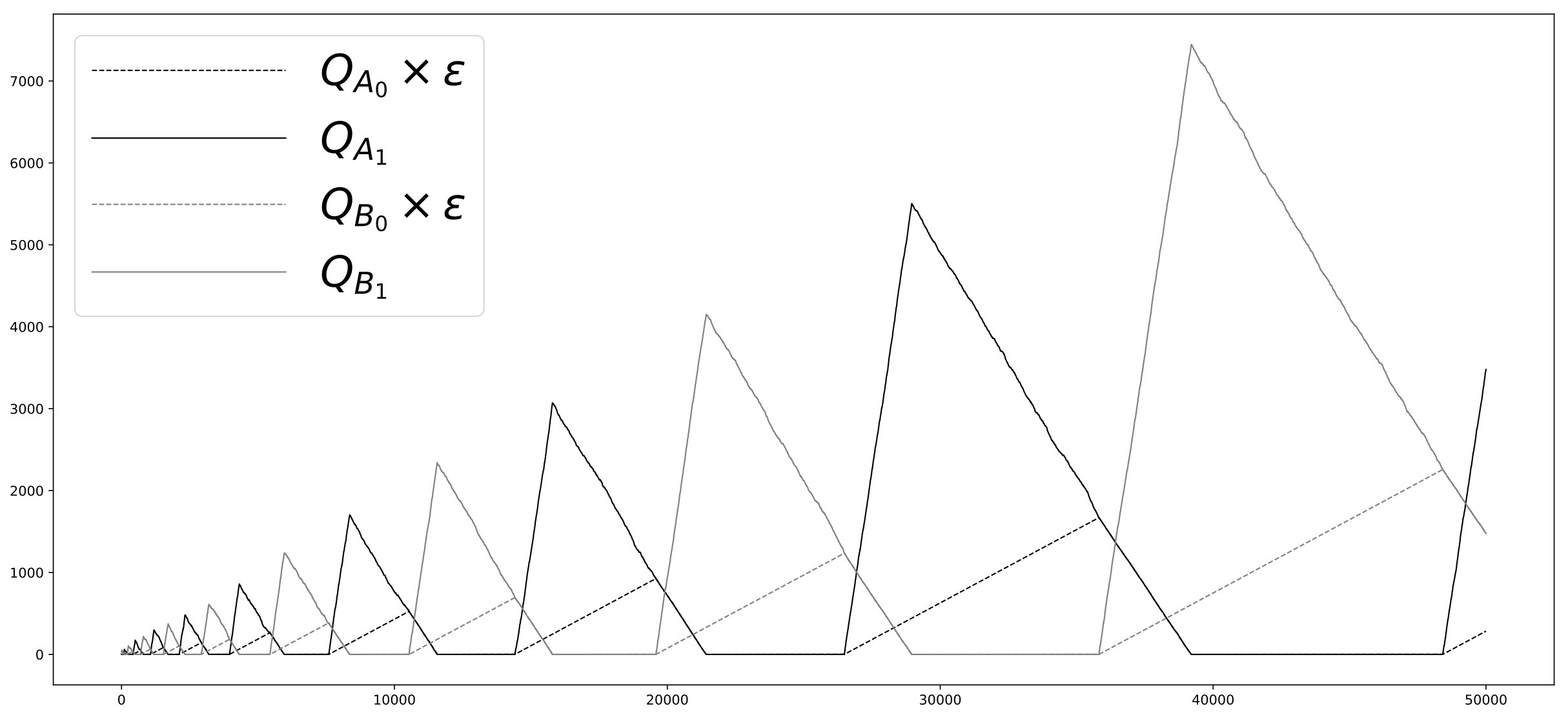}
		\caption{A simulation of the ``collapsed'' analog of the network simulated in Figure \ref{Fig7}
		for the same parameters, with $m = 20\, (= K)$.
		The initial value is the same, with $Q_{\mA_0}(0)= 55$ and all other queues empty. 
		\label{Fig8}}
	\end{minipage}
\end{figure}

Figure \ref{Fig7} shows a simulation for the switched queueing network in Figure \ref{Fig6}. 
The simulation is given for parameters $a = 1$, $\epsilon = 0.1791$ and $K = 20$. 
These simulations strongly suggest that the queueing network is unstable.

Figure \ref{Fig8} shows the evolution of the collapsed analog of the network simulated in Figure \ref{Fig7}.
The mean service times in $\mA_1$ and $\mB_1$ are equal to $m = 20$.
In this case as well, simulations strongly suggest that the single class switched queueing network is unstable.

For the reader interested in more carefully comparing Figure \ref{Fig7} and Figure \ref{Fig8}, note that the vertical scale in Figure \ref{Fig7} is approximately double that in Figure \ref{Fig8}, and consequently the loss of jobs from the $\mA_1$ and $\mB_1$ queues past their peak sizes during a cycle is also approximately twice as quick.  Both are due to the much shorter amount of service required by a job for individual classes of the $\mA_1$ and $\mB_1$ queues in the multiclass network relative to the service required by a job of the $\mA_1$ and $\mB_1$ queues in the collapsed analog ($1$ unit of time versus $(1-2\epsilon)/\epsilon^2$ units of time).  This makes it possible for large numbers of partially served jobs in the multiclass network to quickly complete their service at the $\mA_1$ and $\mB_1$ queues.


\subsection{A Family of Non-FIFO Stable Multiclass Switched Networks.}
\label{stableMC}
In this subsection, we show positive recurrence for multiclass switched networks under a modification of the weighted MaxWeight policy.
A severe drawback of this approach is that knowledge of the number of jobs in each class, rather than in each queue, is required.


The  \emph{largest-class weighted MaxWeight} optimization problem is given by:
\begin{align}\label{eq:MultiMax}
&\text{maximize}\quad\sum_{j\in\mJ} Q^*_j \sigma_j  \quad
\text{over} \quad  {\vecsigma} \in {\mS}\,,
\end{align}
where  
\begin{equation}\label{eq:MultiMax2}
Q^{*}_j := \max_{k\in \mCl(j)} \tilde{Q}_{k}/\tilde{\rho}_{k}\,.
\end{equation}
This optimization maximizes the objective function with respective to the weighted largest class $Q^{*}_j$ at queue $j$, rather than with respect to $Q_j$.  The entire work at queue $j$ is then devoted to jobs of the class $k\in \mCl (j)$ where $\tilde{Q}_{k}/\tilde{\rho}_{k}$  is maximized.

\begin{theorem}
Consider the largest-class weighted MaxWeight policy for a multiclass switched network.  If $\bm {\rho} \in \mC$, then the associated queueing network process $\bm{\mathcal{Q}}$ is positive recurrent.
\end{theorem}

\begin{proof}
The result is an elementary consequence of Theorem  \ref{thrm:p-MW}. Reinterpret each class $k$, $k\in \mK$, as an individual queue within a single class switched queueing network. The routing of jobs is given, as before, by the matrix $(P_{k_1 k_2})_{k_1,k_2 \in \mK}$. The set of feasible schedules is given by 
$\tilde{\bm \sigma} \in \tilde{\mS}$, $\tilde{\bm \sigma} = (\tilde{\sigma_1},\ldots, \tilde{\sigma}_K)$,  where $\tilde{\sigma_k}$ is constrained by $\sigma_j \in \mS$, with $\mS$ is as in (\ref{eq:MultiMax}), and by
$\sigma_j = \sum_{k\in \mCl (j)}\tilde{\sigma}_k$.


The weighted MaxWeight policy (\ref{eq1wtmaxwt}) for this single class switched queueing network is the same as the policy (\ref{eq:MultiMax})--(\ref{eq:MultiMax2}) for the original multiclass switched queueing network, and therefore the two policies generate the same queueing network process, after the above reinterpretation of classes as queues.
The desired result therefore follows by applying Theorem \ref{thrm:p-MW}.
%
\end{proof}


%

\section{Comparison of MaxWeight with the ProportionalScheduler Policy.}\label{sec:Conc}

%


The \textit{ProportionalScheduler Policy} is defined in the same way as MaxWeight, but with the optimization problem (\ref{eq1}) replaced by
\begin{align}
\label{eq:PS}
&\text{maximize}\quad\sum_{j\in\mJ} Q_j \, {\log \sigma_j}\quad
\text{over} \quad  {\vecsigma} \in <{\mS}>\ ,
\end{align}
where $<{\mS}>$ is the convex hull of $\mS$.  When such a solution is a weighted average of schedules in $\mS$, a schedule is chosen at random according to these weights.
It was shown in \cite{BDW17} that the ProportionalScheduler is maximally stable for multiclass multi-hop switched queueing networks of Kelly type; it is consequently also maximally stable for single class multi-hop switched queueing networks. 

Why is maximum stability so much more robust for the ProportionalScheduler than for MaxWeight?  
This can be paraphrased as asking why no terms corresponding to $\rho_j$ in the weighted MaxWeight optimization (\ref{eq1wtmaxwt}) are needed in the ProportionalScheduler setting (\ref{eq:PS}) for general mean service times $m_j$ for fixed $\lambda_j$. 

To explain why the terms $\rho_j$ are not needed in (\ref{eq:PS}), we will consider why the ProportionalScheduler is maximally stable for arbitrary choices of $m_j$, given that it is maximally stable for $m_j \equiv 1$.  For this, we compare the optimization problem for $m_j \equiv 1$ on $<S>$, with the optimization problem for arbitrary $m_j$ on $<S^m>$, where
\begin{equation} 
\label{equationforS}
S^m := \{(m_1 \sigma_1, \ldots, m_{|\mJ|} \sigma_{|\mJ|}) \,: \, {\bm \sigma} \in \mS \} 
\quad \text{for } {\bm \sigma} = (\sigma_1,\ldots,\sigma_{|\mJ|})\,,
\end{equation}
i.e., $\mS$ is stretched by the factor $m_j$ in each coordinate $j$.   Since, for $\sigma_j^m  := m_j \sigma_j$,
\begin{equation*} 
\sum_{j\in \mJ} Q_j \log \sigma_j^m = \sum_{j\in \mJ} Q_j \log \sigma_j  + Q({\bf m}) \,,
\end{equation*}
where $Q({\bf m}) := \sum_{j\in \mJ} Q_j \log m_j$ does not depend on ${\bm \sigma}$, it is easy to check that $\hat {\bm \sigma}^m = (\hat \sigma_1^m,\dots, \hat \sigma_{\mJ}^m)$ maximizes $\sum_{j\in\mJ} Q_j \, {\log \sigma_j^m}$
over ${\vecsigma^m} \in <{\mS^m}>$ exactly when $\hat {\bm \sigma} = (\hat \sigma_1,\dots, \hat \sigma_{\mJ})$ maximizes $\sum_{j\in\mJ} Q_j \, {\log \sigma_j}$
over ${\vecsigma} \in <{\mS}>$.   

On the other hand, service at each queue $j$ is completed at rate $1/m_j,$ per unit work, for the queueing network process ${\bf Q}^m (t)$.  Since $\hat \sigma_j^m / \hat \sigma_j = m_j$, the queueing network processes  ${\bf Q} (t)$ and ${\bf Q}^m (t)$ complete service at the same rate.  Consequently, ${\bf Q} (t)$ and ${\bf Q}^m (t)$ will  have the same evolution, except for the presumably minor difference due to the different discretizations needed for choosing a schedule on $\mS^m$ as opposed to $\mS$, and hence both or neither process should be positive recurrent.
 So, maximal stability in the $m_j \equiv 1$ setting should imply maximal stability for arbitrary $m_j$.

In contrast, the positive recurrence of a subcritical switched queueing network process under MaxWeight depends on the choice of $m_j$.  For example, the switched queueing network in Theorem \ref{thrm:MW} satisfying (\ref{superqueueload}) and (\ref{thm1condition}) is transient. Applying (\ref{equationforS}) to scale the set $S$ and mean service times $m_j$, so that $\rho_j \equiv 1$, now produces a switched queueing network process that is positive recurrent under MaxWeight, by Theorem \ref{thrm:p-MW}.  This scaling also explains the role the terms $\rho_j$ play in the weighted MaxWeight optimization (\ref{eq1wtmaxwt}), since scaling both $\mS$ and the objective function by $\rho_j$ produces the same solutions for the optimization problem as without this scaling.

\strut \\
\noindent \textbf{Acknowledgments.} The research of the M. Bramson was partially supported by NSF grant DMS-1203201. The research of the D'Auria was partially supported by the Spanish Ministry of Economy and Competitiveness Grants [MTM2017-85618-P via FEDER funds]. Part of this research was done while he was a visiting professor at the NYUAD (Abu Dhabi, United Arab Emirates).

\bibliographystyle{amsplain} 
\bibliography{REFERENCES} 

 \appendix


\section{Convergence to Fluid Models.} \label{fluidmodelconvergence}

In Proposition \ref{FluidLimit}, we derive the fluid model equations, given in (\ref{eq:FN})--(\ref{argmax}) of Section \ref{sec:Fluid_Model}, from scaled limits of the queueing network equations of switched queueing networks, operating under the weighted MaxWeight policy. 
Our approach follows the standard fluid limit approach in \cite{Br08} and  \cite{Da95}, and is similar to that in  Appendix C of \cite{BDW17}. 

Let ${\procQ^c}=(\vecQ^c(t):\, t\in \bZ_+)$, $c\in \bN$,
be a sequence of weighted MaxWeight queueing network processes; we
employ
notation $A^c_j$, $D^c_j$, $E^c_j$, $Q^c_j$, $S^c_j$,  $\Phi^c_{jj'}$ and $\Pi^c_j$, for $j,j'\in \mJ$, analogous 
to that introduced in Section \ref{sec:model}.  The evolution of the processes is assumed to be identical, except for their initial states, which satisfy
$c=|\vecQ^c(0)| = \sum_{j\in\mJ} Q^c_j(0)$,
and the processes are coupled on the same probability space, with external arrival, job size, and routing processes each being the same for different $c$:  
\begin{equation*}
E^c_j(t) =  E_j(t),
\quad
S^c_j(t) =  S_j(t),
\quad \text{and}\quad 
\Phi^c_{jj'}(t) = \Phi_{jj'} (t),\qquad j,j'\in \mathcal J, \,\, c\in\bN\,.
\end{equation*}
For $j,j'\in \mathcal J$, we introduce the 
scaled processes
\begin{subequations}\label{eq:fluid.scaled.processes}
\begin{align}
\bar{A}^c_j(t) =& A_j^c( ct )/c,\;\;\bar{D}^c_j(t) = D_j^c(ct)/c,\;\;\bar{E}^c_j(t) = E_j^c(ct)/c,\;\; \bar{Q}^c_j(t) = Q^c_j(ct)/c\, ,
\\[.35em]
&\bar{S}^c_j(t) = S^c_j(ct)/c,\;\; \bar{\Phi}^c_{jj'}(t) = \Phi_{jj'}^c( ct )/c,\;\;\bar{\Pi}^c_{j}(t) = \Pi_{j}^c(ct)/c,
\end{align}	
\end{subequations}
for $t\in\{0,c^{-1},2c^{-1}, 3c^{-1},...\}$, and we interpolate linearly for other values of $t\in\mathbb R_+$.

As $c\rightarrow \infty$, these scaled processes converge to 
the fluid model equations \eqref{eq:FN}--\eqref{argmax} in the following sense:

\begin{proposition}\label{FluidLimit}
	There exists a set $G$ with $\mathbb P( G)=1$ such that, for all $\omega \in G$,
	any scaled subsequence  $(\bar{A}^{c_i}_j, \bar{D}^{c_i}_j, \bar{E}^{c_i}_j, \bar{Q}^{c_i}_j, \bar{S}^{c_i}_j, \bar{\Phi}^{c_i}_{jj'}, \bar{\Pi}^{c_i}_j: j,j'\in \mJ)$, $c_1<c_2<\ldots$, of switched networks under the weighted MaxWeight policy, contains a further subsequence that converges uniformly on compact time intervals. 
	 Moreover, any such limit $(\bar{A}_j, \bar{D}_j, \bar{E}_j, \bar{Q}_j, \bar{S}_j, \bar{\Phi}_{jj'}, \bar{\Pi}_j: j,j'\in \mJ)$ is a Lipschitz continuous process satisfying the weighted MaxWeight fluid model equations (\ref{eq:FN})--(\ref{argmax}). 
\end{proposition}

Note that, since almost sure convergence implies convergence in distribution, Proposition \ref{FluidLimit} implies tightness/relative compactness and characterizes the weak convergent limits of the sequence $\{ (\bar{A}^{c}_j, \bar{D}^{c}_j, \bar{E}^{c}_j, \bar{Q}^{c}_j, \bar{S}^{c}_j, \bar{\Phi}^{c}_{jj'}, \bar{\Pi}^{c}_j: j,j'\in \mJ) \}_{c\in \bN}$.  

The proof of Proposition \ref{FluidLimit} is  straightforward with the exception of showing the subsequential limits of
$ (\bar{\bm{\Pi}}^{c}, \bar{\bm{Q}}^{c})$
 satisfy (\ref{argmax}), which is rather tedious.

\smallskip


\begin{proof}[{Proof of Proposition \ref{FluidLimit}.}]

In order to demonstrate Proposition \ref{FluidLimit}, we recall that, for each $j,j'\in \mathcal J$,  
\begin{equation*}
(E_j(t)-E_j(t-1): t\in \mathbb N), \quad
(S_j(t)-S_j(t-1): t\in \mathbb N)\,,
\quad
\text{and}
\quad 
(\Phi_{jj'}(t)-\Phi_{jj'}(t-1): t\in \mathbb N) 	
\end{equation*}
 are collections of i.i.d. random variables with respective  
means $a_j$, $p_j^{-1}$ and $P_{jj'}$. Therefore, by
the (Functional) Strong Law of Large Numbers, on a set $G_1$ with $\mathbb{P}(G_1) = 1$, 
\begin{equation}\label{ESPhi}
\bar{E}^c_j(t)
\rightarrow 
a_j t\ ,
\quad
\bar{S}^c_j(t)
\rightarrow
p_j t \,,
\quad
\text{and}
\quad 
\bar{\Phi}^c_{jj'}(t)
\rightarrow 
P_{jj'} t	
\end{equation}
as $c\rightarrow\infty$, for $j,j'\in \mathcal J$, with convergence being uniform on compact time intervals.

By the Arzel\`a-Ascoli Theorem (cf. \cite{Bi99}), any sequence of equicontinuous functions $\bar{X}^{c_i} (t)$ on $[0,T]$, $T\in (0, \infty)$, with $\sup_{c_i} |\bar{X}^{c_i} (0)| < \infty$, has a converging subsequence with respect to the uniform norm.  
We will show that the sequences $\{  (\bar{A}^{c}_j, \bar{D}^{c}_j, \bar{E}^{c}_j, \bar{Q}^{c}_j, \bar{S}^{c}_j, \bar{\Phi}^{c}_{jj'}, \bar{\Pi}^{c}_j: j,j'\in \mJ)\}_{c\in\bN}$ satisfy both conditions for any $\omega \in G_1$. 

Since $|\bar{\vecQ}^{c_i}(0)|=1$ and the other functions are initially $0$, the supremum is clearly 
bounded.  In order to show equicontinuity, we first note that the sequences $\bar{E}^c_j$, $\bar{S}^c_j$, $\bar{\Phi}^c_{jj'}$ convergence uniformly on compact sets by the Functional Strong Law, and so these functions are equicontinuous.

 To show equicontinuity for the remaining functions, 
note that ${D}^{c}_j$ and  ${\Pi}^{c}_j$ both have bounded increments, and so their rescaled analogs are uniformly Lipschitz continuous;
this implies equicontinuity of the corresponding sequences. Since $\bar{Q}^{c}_j$ and $\bar{A}^c_j$ are the sum of a bounded number of such equicontinuous functions, they too are equicontinuous. Therefore, since the conditions of the Arzel\`a-Ascoli Theorem are met for $\omega \in G_1$, every subsequence $\{  (\bar{A}^{c_i}_j, \bar{D}^{c_i}_j, \bar{E}^{c_i}_j, \bar{Q}^{c_i}_j, \bar{S}^{c_i}_j, \bar{\Phi}^{c_i}_{jj'}, \bar{\Pi}^{c_i}_j: j,j'\in \mJ)\}_c$ has a further subsequence that converges uniformly on $[0,T]$.  Moreover, since these sequences of functions are uniformly Lipschitz, so are their limits.


We need to show that the fluid model equations in  (\ref{eq:FN})--(\ref{argmax}) are satisfied for all such limits. The equations in  (\ref{eq:FN}) follow directly from the queueing network equations in (\ref{eq:QN}) and from \eqref{ESPhi}.
So, it remains to show that \eqref{argmax} holds. 

For $\omega \in G_1$, consider a sequence 
$ (\bar{\bm{\Pi}}^{c_i}, \bar{\bm{Q}}^{c_i})$ that converges uniformly on compact sets to $ (\bar{\bm{\Pi}}, \bar{\bm{Q}})$.  
In order to show  $ (\bar{\bm{\Pi}}, \bar{\bm{Q}})$ satisfies \eqref{argmax}, we first introduce $\hat{\Pi}^c_j(t)$, where
$\hat{\Pi}^c_j(i/c) = \frac{1}{c} \sum_{\ell=1}^{i} \sigma_j(\vecQ^c(\ell))$, for 
$i \in \bZ_+$, and interpolate for other values of $t$.
 We note that, for $t>s$, $\bar{\Pi}_j^c(t) -  \bar{\Pi}_j^c(s) \le  \hat{\Pi}_j^c(t) -  \hat{\Pi}_j^c(s)$, with the difference only being nonzero due to underutilization  of work capacity at $j$, which can only occur when  $\bar{Q}^c_j (u) < \sigma_{\max}$ for some $u \in [s-1/c,t]$,
 where $\sigma_{\max}$ is the maximal value attainable by any coordinate of $\bm\sigma$ in $\mS$.  We will find it more convenient to work with $\hat{\bf\Pi}$ than with $\bar{\bf\Pi}$.

It follows from the definition of $\hat{\bm{\Pi}}^c$ that
$(\hat{\bm{\Pi}}^c(t) -  \hat{\bm{\Pi}}^c(s))/(t-s) \in < \bar{\mS} >
$.  Since $\bar{\bm{\Pi}}^c(t) -  \bar{\bm{\Pi}}^c(s) \le  \hat{\bm{\Pi}}^c(t) -  \hat{\bm{\Pi}}^c(s)$,
%
the above  limit $ \bar{\bf{\Pi}}$ satisfies
\[
\frac{ \bar{\bm{\Pi}}(t) -  \bar{{\bm{\Pi}}}(s)}{t-s} \in < \bar{\mS} > \quad \text{ for any } t>s.
\]
So, $\bar{\bm{\Pi}}'(t) \in < \bar{\mS} >$ wherever the derivative exists. 

We still need to show that $\bar{\bm{\Pi}}'(t)$ solves the MaxWeight optimization over $< \bar{\mS} >$.
We will employ the limit 

\begin{equation}\label{Qutipi---1}
\int_s^t 
	\frac{\bar{\bm{Q}^{c}}(u) }{\bm{\rho}}
	\cdot 
	\hat{\bm{\Pi}}^c(du)
- 
\int_s^t 
	\frac{\bar{\bm{Q}^{c}}(u) }{\bm{\rho}}
	\cdot 
	\bar{\bm{\Pi}}^c(du)	
\xrightarrow[c\rightarrow\infty ]{}  0 \quad \text{ for any } t>s \,,
\end{equation}
where
$\frac{\bm{Q}}{\bm{\rho}} := \Big( \frac{{Q_j}}{{\rho_j}} : j\in \mathcal J \Big)$.  
In order to show (\ref{Qutipi---1}), we observe that
\begin{align*}
	0 
	&
	\leq
	\int_s^t 
		\bar{Q}^c_j(u) \hat \Pi^c_j(du) 
	-
	\int_s^t 
		\bar{Q}^c_j(u) \bar \Pi^c_j(du) 	 
	\\[.1em]
	&
	\le
	\frac{1}{c^2}
	\sum_{i=1}^{\lceil ct \rceil}
		\Big[Q_j^c(i) + a_j + |\mathcal{J}|\sigma_{\max}\Big] \sigma_j( \bm Q^c(i)) 
	-
	\frac{1}{c^2}
	\sum_{i=1}^{\lceil ct \rceil}
		Q_j^c(i) \Big[ \sigma_j( \bm Q^c(i)) \wedge Q_j^c(i) \Big]
	\\[.2em]
	&
	\le
	\frac{1}{c^2}
	\sum_{i=1}^{\lceil ct \rceil}
	Q_j^c(i)
	\left[
	 	\sigma_j( \bm Q^c(i))  - Q_j^c(i)
	\right]
	\mathbb I
		[ Q_j^c(i) < \sigma_j(\bm Q^c(i)) ]
		+ (a_j \sigma_{\max} + |\mathcal{J}|\sigma_{\max}^2) 2t/c
	\\[.2em]
	&
	\leq 
	(\sigma_{\max}^2 + a_j \sigma_{\max} + |\mathcal{J}|\sigma_{\max}^2) 2t/c
	\xrightarrow[c\rightarrow \infty ]{} 0 \,,
\end{align*}
from which (\ref{Qutipi---1}) follows.  The quantity $a_j +|\mathcal{J}|\sigma_{\max}$ on the second line is due to the linear interpolation used to define $\bar Q_j^c (u)$, and the first term on the last line includes a factor $\sigma_{\max}$ that bounds $Q_j^c(i)$ on the indicator set of the previous line.

By the optimality of the weighted MaxWeight policy, for any policy $\bar{\vecsigma} \in \mS$ and $t >s$,
\begin{equation}\label{Qutipi-0} 
\int_s^t 
\frac{
\bar{\vecQ}^c(u)
}{\bm{\rho}}
\cdot \hat{\bm{\Pi}}^c(du) 
\geq 
\frac{1}{c^2}
\sum_{i=\lfloor cs \rfloor +1}^{\lfloor ct \rfloor}  
\frac{\vecQ^c(i)}{\bm{\rho}}\cdot {\vecsigma} (\vecQ^c(i)) 
- \frac{\beta_{s,t}}{c}
\ge
\frac{1}{c^2}
\sum_{i=\lfloor cs \rfloor +1}^{\lfloor ct \rfloor}  
\frac{\vecQ^c(i)}{\bm{\rho}}\cdot \bar{\vecsigma} 
- \frac{\beta_{s,t}}{c}\,;
\end{equation}
the term $\beta_{s,t} :=  |\mathcal{J}|\sigma_{\max}^2 (t-s)/\rho_{\min}$ is due to linear interpolation, where $\rho_{\min} := \min_{j\in \mathcal J} \rho_j$.
We claim that
\begin{equation}\label{pictopi}
\int_s^t 
	\frac{\bar{\bm{Q}}^c(u) }{\bm{\rho}}
	\cdot 
	\hat{\bm{\Pi}}^c(du)
\xrightarrow[c\rightarrow\infty ]{}  
\int_s^t 
\frac{\bar{\bm{Q}}(u) }{\bm{\rho}}
\cdot 
\bar{\bm{\Pi}}(du)
\end{equation}
and 
\begin{equation}\label{sequenceconvergence}
\frac{1}{c^2}
\sum_{i=\lfloor cs \rfloor +1}^{\lfloor ct \rfloor}  
\frac{\vecQ^c(i)}{\bm{\rho}}\cdot \bar{\vecsigma}
\xrightarrow[c\rightarrow\infty ]{}  
\int_s^t 
\frac{\bar{\bm{Q}}(u) }{\bm{\rho}}
\cdot 
\bar{\bm{\sigma}} \, du\,.
\end{equation}
The limit  (\ref{sequenceconvergence}) follows immediately from the uniform convergence of $\bar{\bm{Q}}^c$ to $\bar{\bm{Q}}$ on compact sets.  To show (\ref{pictopi}), we note that, by the uniform convergence on compact sets of 
$\bar{\bm{Q}}^c$ to $\bar{\bm{Q}}$ and  $\bar{\bm{\Pi}}^c_j$ to 
$\bar{\bm{\Pi}}_j$,
\begin{equation}\label{before93}
\left| 
	\int_s^t 
		\Big[ 
		\frac{\bar{\bm{Q}}^c(u) }{\bm{\rho}}
		- 
		\frac{\bar{\bm{Q}}(u)}{\bm{\rho}} 
		\Big] \cdot
		\bar{\bm{\Pi}}^c(du) 
\right| 
\leq 
\frac{
|\mathcal J|\sigma_{\max} 
}{
\rho_{\min}
}
|| 
\bar{\bm{Q}}^c - \bar{\bm{Q}} ||_\infty  (t-s)\xrightarrow[c\rightarrow \infty]{} 0
\end{equation}
and
\begin{equation}\label{Qutipi-1}
\int_s^t 
	\frac{\bar{\bm{Q}}(u) }{\bm{\rho}}
	\cdot 
	\bar{\bm{\Pi}}^c(du)
\xrightarrow[c\rightarrow\infty ]{}  
\int_s^t 
\frac{\bar{\bm{Q}}(u) }{\bm{\rho}}
\cdot 
\bar{\bm{\Pi}}(du)\,,
\end{equation}
where
$|| \cdot ||_\infty$ denotes the uniform norm over all coordinates.
The limit (\ref{pictopi}) is an immediate consequence of (\ref{Qutipi---1}), (\ref{before93}), and (\ref{Qutipi-1}).

It follows immediately from (\ref{Qutipi-0})--(\ref{sequenceconvergence}) that, for any policy $\bar{\vecsigma} \in \mS$ and $t >s$,
\begin{equation} \label{optimality}
\int_s^t 
\frac{\bar{\bm{Q}}(u)}{\bm{\rho}}
\cdot \bar{\bm{\sigma}} du
\leq 
\int_s^t 
\frac{\bar{\bm{Q}}(u)}{\bm{\rho}} 
\cdot \bar{\bm{\Pi}}(du) \,.
\end{equation}
Differentiating both sides of (\ref{optimality}) with respect to $t$ implies that
$
\frac{\bar{\bm{Q}}(t)}{\bm{\rho}} 
\cdot \bar{\bm{\sigma}} 
\leq 
\frac{\bar{\bm{Q}}(t)}{\bm{\rho}} 
\cdot \bar{\bm{\Pi}}'(t) ,
$
wherever $\bar{\bm{\Pi}}'(t)$ exists.
Since $\bar{\bm{\Pi}}'(t) \in <\mathcal S>$, it follows that 
$
\bar{\bm{\Pi}}'(t) \in  \argmax_{{\vecsigma} \in {\mS}}\; \sum_{j\in\mJ} \bar{Q}_j(t) \sigma_j / \rho_j$, as desired.
\end{proof}

%
%

\smallskip

\section{Fluid Stability Implies Positive Recurrence.}\label{Sec:PosRec}
Here, we demonstrate Proposition \ref{propFMQN} of Section \ref{sec:Fluid_Model}.  As mentioned in that subsection, the use of fluid stability is a standard approach for showing positive recurrence of queueing networks.

\begin{proof}[{Proof of Proposition  \ref{propFMQN}.}]

One can check that, for each $t\geq 0$, the sequence of queue sizes 
$\{ |\bar{\vecQ}^{c}(t)| \}_{c \in \bN}$ 
of the
scaled proportional switched networks in Proposition \ref{FluidLimit}
is uniformly integrable. This follows quickly from the inequality
\begin{equation}
\label{equnifint}
|\bar{\vecQ}^{c}(t)| =\sum_{j\in\mJ} \bar{Q}_j^{c}(t)\leq \sum_{j\in\mJ} \frac{Q^{c}_j(0)}{c} + \sum_{j\in\mJ} \frac{E^{c}_j(t)}{c}\,,
\end{equation}
since $E^{c}_j(t)$ is a sum of i.i.d. random variables with finite mean (see, e.g., \cite[Lemma 4.13, (4.81)]{Br08}). 


On the other hand, by Proposition \ref{FluidLimit}, on a set $G$ with $P(G)=1$,  every subsequence of
$\bar{\vecQ}^{c}(t)$ has a further subsequence that converges uniformly on compact time
 intervals to a fluid model solution $\vecQ(t)$ of (\ref{eq:FN})--(\ref{argmax}), with $|\vecQ(0)|=1$. Since the fluid model is assumed to be stable, all fluid model solutions with $|\vecQ(0)|=1$  satisfy $|\vecQ(t)| = 0$ for $t\ge t_0$, with $t_0$ not depending on the particular
fluid model solution. 
Reapplying this reasoning along any subsequence of $|\bar{\vecQ}^{c}(t_0)|$ shows that, in fact, $|\bar{\vecQ}^{c}(t_0)|$ converges to $0$ along the entire sequence.

By the above uniform integrability and almost sure convergence results on $|\bar{\vecQ}^{c}(t_0)|$, it follows that
\begin{equation*}
\lim_{c\rightarrow\infty} \bE |\bar{\vecQ}^{c}(t_0)|  =0\,,
\end{equation*}
which implies 
$\bE |\vecQ^{c}(ct_0)|   \le c/2$ for large enough $c$.
Since the set of values in the state  space satisfying $\{ \bm Q_0 : | \bm Q_0 | \leq c \}$ is finite for all $c$, one can apply a generalized version of Foster's criterion to conclude that the process $\bm Q(t)$ is positive recurrent.  (See
 Robert \cite[Theorem 8.6]{robert2013stochastic}; or, in the more general continuous time setting, 
  \cite[Proposition 4.5]{Br08}.)
This concludes the proof of Proposition \ref{propFMQN}.
\end{proof}

\section{Positive Recurrence in Switched Networks with Pure Branching.}

Here, we prove Theorem \ref{purebranching}, which demonstrates positive recurrence for switched networks with pure branching.  The proof employs fluid models, as well as a Lyapunov function $g(t)$ that is similar to the Lyapunov function $h(t)$ of Theorems \ref{thrm:p-MW} and \ref{tandemtheorem}.  Interestingly, this choice of Lyapunov function does not work for general network topologies, unlike the Lyapunov function $h(t)$ (although it is equivalent to $h(t)$ in the restricted setting of tandem networks in Theorem \ref{tandemtheorem}).  Conversely, the Lyapunov function $h(t)$ cannot be employed for Theorem \ref{purebranching}.

 In the proof of the theorem, $b(j) \in \mathcal J$ denotes the queue preceding the queue $j$, provided such a queue exists, and $\mathcal J^b \subset \mathcal J$ denotes those queues $j$ for which $b(j)$ exists.


\smallskip

\begin{proof}[{Proof of Theorem \ref{purebranching}.}]
We define the following Lyapunov function
\begin{equation*}
g(t): 
= 
\max_{{\vecsigma} \in {\mS}}\; \sum_{j\in \mathcal J} Q_{j}(t) 
\left( 
{\sigma_{j}}
- \rho_{j}
\right)
=
\sum_{j\in \mathcal J}
Q_{j}(t) 
\big({D'_{j}(t)}- \lambda_{j} \big) \, ,
\end{equation*}
where $\rho_{j}$ is the traffic intensity at queue $j$, $\lambda_j$ is the total arrival rate there, and $Q_j(t)$ and $D_j(t)$ satisfy the fluid model equations (\ref{eq:FN})--(\ref{argmax}). 
 Because $\bm{\rho} \in\mathcal C$ is assumed,  $g(t)> 0$ when $\vecQ(t) \neq 0$.

As in Theorems \ref{thrm:p-MW} and \ref{tandemtheorem}, one can show that 
\begin{equation} \label{boundong}
g'(t)
\leq 
\sum_{j\in \mathcal J} 
Q_j'(t) 
\Big( D'_j(t) - \lambda_j \Big)\, .
\end{equation}
The right-hand side of (\ref{boundong}) can be bounded as follows:
\begin{align*}
\sum_{j\in \mathcal J} 
Q_j'(t) 
\Big( D'_j(t) - \lambda_j \Big)\
& 
=
\sum_{j\in \mathcal J} 
\Big( A'_j(t)- D'_j(t) \Big)
\Big( D'_j(t) -A'_j(t) + A'_j(t) - \lambda_j \Big)
\\[.2em]
=
&
-\frac{1}{2} \sum_{j\in \mathcal J}
\Big( A'_j(t) - D'_j(t) \Big)^2
- 
\frac{1}{2}
\sum_{j\in \mathcal J}
\Big(D'_j(t) -   \lambda_j \Big)^2
+
\frac{1}{2}
\sum_{j\in \mathcal J}
\Big( A'_j(t) -  \lambda_j \Big)^2
\\[.1em]
=
&
-
\frac{1}{2}
\sum_{j\in \mathcal J}
\Big( A'_j(t) - D'_j(t) \Big)^2
-
\frac{1}{2}
\sum_{j\in\mathcal J} 
\Big( 
D_j'(t) - \lambda_j
\Big)^2
+
\frac{1}{2}
\sum_{j \in \mathcal J^b} 
P_{b(j)j}^2\Big( D'_{b(j)}(t) -  \lambda_{b(j)} \Big)^2
\\[.1em]
=
&
-
\frac{1}{2}
\sum_{j\in \mathcal J}
\Big( A'_j(t) - D'_j(t) \Big)^2
-
\frac{1}{2}
\sum_{j\in\mathcal J} 
\bigg[
1-\!\!\!  \sum_{j' : b(j')=j} \!\!\!     P_{j j'}^2 
\bigg]
\Big( 
D_j'(t) - \lambda_j
\Big)^2
\\[.1em]
\leq 
&
-\frac{1}{2} \sum_{j \in \mathcal J} 
	\Big(
		A_j'(t) - D_j'(t)
	\Big)^2.
\end{align*}
In the preceding display, the first equality follows from the definition of $Q_j$ and the third equality employs
\[
A'_j(t) = D'_{b(j)} (t) P_{b(j) j } ,
\quad
\lambda_j = \lambda_{b(j)} P_{b(j) j } \,\,; 
\]
 note that, if no queue precedes $j$,  then $j$ instead has exogenous arrivals, and so $A_j'(t)=\lambda_j$.
Combining the above inequalities yields the bound
\begin{equation}
g'(t)
\leq 
-\frac{1}{2} \sum_{j \in \mathcal J} 
	\Big(
		A_j'(t) - D_j'(t)
	\Big)^2.
\label{g:bound}	
\end{equation}

The remainder of the argument proceeds similarly to that for Theorem \ref{tandemtheorem}. 
We first note that, if $|A'_j(t) - D'_j(t)| \leq \epsilon$ for all $j\in\mathcal J$, 
then 
$ |D'_j(t) - \lambda_j| \leq |\mathcal J |\epsilon $ for all $j$, since there is a sequence of queues at most $|\mathcal J|$ long connecting any queue with the external arrival stream.
On the other hand, $({D}'_j(t) : j \in \mathcal J)$ belongs to the boundary of the set $<\mathcal S>$, whereas $\bm{\rho}$ does not. So  $| D'_j(t) - \lambda_j| > | \mathcal J | \epsilon $ for some $\epsilon > 0$. Therefore, for some $j\in \mathcal J$, 
\begin{equation} \label{ADinequality}
| A_j'(t) - D'_j(t) | > \epsilon.
\end{equation}
Together with \eqref{g:bound}, (\ref{ADinequality}) implies that, for all times $t$ with $\bm Q(t) \neq 0$,
\[
g'(t) \leq -\frac{1}{2} \epsilon^2.
\]
It immediately follows that, for $t \ge t_0 :=2|\mathcal J| g(0)/\epsilon^2$, $g(t)=0$.  Hence $\bm Q(t)=0$, which implies the corresponding fluid model is stable.  By Proposition \ref{propFMQN}, the switched queueing network is therefore positive recurrent.
\end{proof}

\section{Transience of a Subcritical LQFS Multiclass Queueing Network.}  \label{appendixD}

A multiclass queueing network is a network that permits more than one class of job at a station; this terminology was popularized by Harrison and is now standard in the literature (see, e.g., \cite{Br08}).  In the present context, we equate ``class" with ``queue" and ``station" with ``component", and employ terminology associated with queue and component that has been used throughout this paper.

Here, we briefly discuss the instability of multiclass queueing networks with the Longest-Queue-First-Served (LQFS) policy.  As the name suggests, over each unit of time, the policy selects the longest queue at each station and devotes the maximal allowed amount of service to these queues.  (In the case of a tie among classes, any of these classes may be chosen.)  

The LQFS policy can be rephrased in terms of the MaxWeight optimization problem,
\begin{equation}\label{unweightedmaxweight}
\text{maximize} 
\quad 
\sum_{j=1}^J
Q_j \sigma_j
	\quad \text{over } \vecsigma \in \mS\,,
	\end{equation}
where the set of feasible schedules $ \vecsigma \in \mS$ is given by 	
$\sum_{j=1}^J \sigma_j \leq 1$,
with $\sigma_j$ being the amount of work that can be applied at class $j$.  The LQFS policy can therefore be investigated within the framework of the current paper;
for example, Theorems  \ref{thrm:p-MW} and \ref{tandemtheorem} can be applied.

We present here a subcritical, but transient, LQFS queueing network whose behavior mimics that of the example 
 in Figure \ref{Fig2} and Theorem \ref{thrm:MW} of a subcritical transient switched network.  For our example, we assume batch service, with $\nu_j \in \mathbb{Z}_+$ class $j$ jobs capable of being served in unit time.  Rescaling the amount of work required for each job by $\nu_j$, so that each job now requires one unit of service, the LQFS policy solves the optimization problem
\begin{equation}\label{lqoptim}
\text{maximize} 
\quad 
\sum_{j=1}^J
Q_j \frac{\sigma_j}{\nu_j}
	\quad \text{over } \vecsigma \in \mS'\,,
	\end{equation}
where the set of feasible schedules $ \vecsigma \in \mS'$ is now given by 	
$\sum_{j=1}^J \sigma_j  /\nu_j \leq 1$.  
The policy in (\ref{lqoptim}) is a weighted MaxWeight policy, although with different weighting than in (\ref{eq1wtmaxwt}).

Our example adopts the framework of Figure \ref{Fig2}, with components $\mA$ and $\mB$, queues $\mA_j$ and $\mB_j$, $j=0,\ldots,J$, and parameters $a$, $\nu$, and $J$ that satisfy all of the assumptions in the paragraph containing (\ref{displayforrho})--(\ref{superqueueload}) and in its preceding paragraph.
In particular, (\ref{lqoptim}) in this framework becomes
\begin{equation} \label{specificlqoptim}
\text{maximize} 
\quad 
Q_{\mathcal A_0} \sigma_{\mathcal A_0} 
+ \frac{1}{\nu} \sum_{j=1}^J Q_{\mathcal A_j} 
\sigma_{\mathcal A_j}  \quad
\text{and} \quad
Q_{\mathcal B_0}\sigma_{\mathcal B_0} 
+
\frac{1}{\nu} \sum_{j=1}^J Q_{\mathcal B_j} 
\sigma_{\mathcal B_j}  \,,
\end{equation}
where the set $\mS'$ of feasible schedules now satisfies (\ref{Sfortheorem1}).
The condition (\ref{superqueueload}) implies that the traffic intensity $\vecrho \in \mC$, and also that the corresponding LQFS queueing network is subcritical in the standard sense of multiclass queueing networks (i.e., the traffic intensity at each station is $<1$).

Theorem \ref{thrm:LQF} is the analog of Theorem \ref{thrm:MW}.

\begin{theorem}
\label{thrm:LQF}
 Consider the switched network represented by Figure \ref{Fig2} that solves the optimization problem in (\ref{specificlqoptim}) and (\ref{Sfortheorem1}).   Assume that (\ref{superqueueload}) is satisfied and that 
\begin{align}
\label{thm6condition}
1 < \nu < J  ,
\qquad \text{and} \qquad
\frac{J}{2J - 1} < a < 1-\frac{(J+1)(J+\nu)}{\nu J^2 + J + \nu}  \,.
\end{align}
Then $\vecrho\in \mC$ and the associated queueing network process {$\procQ$} is transient. 
It follows immediately that the corresponding multiclass LQFS queueing network is both subcritical and transient.
\end{theorem}

It is easy to choose $a$, $\nu$, and $J$ so that both (\ref{superqueueload})
and~\eqref{thm1condition} are satisfied.
For instance, one can choose
$a=7/12$, $\nu=6$, and $J=30$, as below Theorem \ref{thrm:MW}. 

We note the following comparison of the lower bound for $a$ given by (\ref{thm6condition}) with that of (\ref{thm1condition}) of Theorem \ref{thrm:MW}, for $\nu = 6$ and $J=30$.
The lower bound in (\ref{thm6condition}) is $30/59$, whereas the corresponding lower bound in (\ref{thm1condition}) is $30/54$.  This lower value of the bound in  (\ref{thm6condition}), as compared with that in (\ref{thm1condition}),  is not surprising since, according to Theorem \ref{thrm:p-MW} and (\ref{eq1wtmaxwt}), the ``correct" denominators $\rho_j$ for $\sigma_{\mA_j}$ and $\sigma_{\mB_j}$, for maximal stability, are  $\rho_0 = 1$ and $\rho_j = 1/30$, for $j=1,\ldots, 30$ (after factoring out $a$), whereas the corresponding denominators in the MaxWeight setting (\ref{eq1}) are both $1$, and the corresponding denominators in (\ref{specificlqoptim}) are $1$ and $6$.  Hence, the ratios of the denominators for the terms with $j=0$ and $j\neq 0$ in the three cases are $30$, $1$, and $1/6$, i.e., the choice of denominators in (\ref{specificlqoptim}) is further than in (\ref{eq1}) from the correct choice, and so one should expect the lowest value of $a$ at which transience of the corresponding process occurs to be less than in (\ref{eq1}).

The proof of Theorem \ref{thrm:LQF} is essentially identical to that of Theorem \ref{thrm:MW}.  One employs Lemma  \ref{HighProb} and Proposition \ref{prop1'} (below), which replaces Proposition \ref{MainProp}.  The statements in Propositions \ref{MainProp} and  \ref{prop1'} are the same except for minor changes.
In Proposition \ref{prop1'}, in lieu of (\ref{InitalCondition}), we will employ (\ref{InitalCondition}$'$),
\begin{subequations}
	\begin{align}
	Q^{\Sigma}_{\mB }(0) & = M \, ,\tag{\ref{InitialB}$'$}
\\[.2em]
	Q^{\Sigma}_{\mA }(0) & \leq \epsilon M / \nu \, ,\tag{\ref{InitialA}$'$}
\\[.15em]
 \left| Q_{\mB_0}(0) -  Q_{\mB_j}(0) \right|  &\leq \epsilon M \, , \text{ for } j=1,...,J \ . \tag{\ref{Gap}$'$}
	\end{align}
\end{subequations}
(The balancing condition (\ref{Gap}$'$) replaces that in (\ref{Gap}).)
In Proposition \ref{prop1'}, in lieu of (\ref{gamma.condition}), we will employ (\ref{gamma.condition}$'$),
\begin{equation}\tag{\ref{gamma.condition}$'$}
1 < \gamma < \frac{a}{1-a+a /J} \,,
\end{equation}
and will again set $M'=\gamma M$. The stopping time $U$ is again defined as in (\ref{eq:tildeT}), but $V$ is modified:
\begin{equation}
U
=
\min\big\{ t \geq 0 : Q_{\mB}^{\Sigma}(t)\leq \nu^2 \big\} \wedge T \,, \quad
V
=
\min
\big\{
t \geq U : Q_{\mA_0}(t) \le  \! \max_{j=1,...,J} Q_{\mA_j}(t)  
\big\} \wedge T \ ,
\tag{\ref{eq:tildeT}$'$}
\end{equation}
where $T$ is given by (\ref{defofT}).

\noindent \textbf{Proposition \ref{MainProp}$'$}. 
\label{prop1'}
	\emph{Assume (\ref{superqueueload}) and (\ref{thm6condition}), and assume (\ref{InitalCondition}$\, '$) with $M$ satisfying \eqref{lw.bound.M}.  Then, on the event $G_M$, }
	\begin{subequations}
		\begin{gather}
		Q^{\Sigma}_\mA  (V) \geq M',\tag{\ref{Init1}$'$}
\\[.2em]
		Q^{\Sigma}_\mB  (V) \leq \epsilon M' / \nu , \tag{\ref{Init2}$'$} 
\\[.12em]
	    \left| Q_{\mA_0}(V) - Q_{\mA_j} (V) \right| \leq \epsilon M' ,\tag{\ref{Init3}$'$}
		\end{gather}
\end{subequations}
\emph{for $j=1,...,J$.  Moreover, for all times $0\leq t\leq V$,}
	\begin{equation}\tag{\ref{Init4}$'$}
	Q_\mA^\Sigma(t) + Q_\mB ^\Sigma(t) \geq \frac{1}{2} \frac{a}{a +\nu}M.
	\end{equation}

Given  Lemma \ref{HighProb} and  Proposition \ref{prop1'},
the proof of Theorem \ref{thrm:LQF} is identical to that of Theorem \ref{thrm:MW}.  

The demonstration of Proposition \ref{prop1'} is very similar to that of Proposition \ref{MainProp}, and employs a sequence of lemmas that are nearly identical to Lemma \ref{LemmaEmpty} -- Lemma \ref{Lemma8}, which we label here as Lemma \ref{LemmaEmpty}$'$ -- Lemma \ref{Lemma8}$'$.  The statement of each of the new lemmas is identical to that of the corresponding lemma employed for Proposition \ref{MainProp}, with the exception of Lemmas \ref{A0 arrivals}$'$, \ref{LemmaStopService}$'$,  and \ref{Lemma8}$'$, where the conditions (\ref{lemma3eq}), (\ref{displayforlemm4}), and  (\ref{eq:Lem8A})--(\ref{eq:Lem8B}) are replaced by 
\begin{equation} \tag{\ref{lemma3eq}$'$}
A_{\mA_0} (t)
\geq\nu (1 - a ) \left(1 - 
  \frac{\nu}{\nu+J}  \right)t
- \epsilon  \kappa_1 M  \, ,
\end{equation}
\begin{equation} 
\tag{\ref{displayforlemm4}$'$}
	Q_{\mA_0} (t)
	>
 \max_{j=1,...,J} Q_{\mA_j}(t)\, ,
	\end{equation}
and 
\begin{align}
&
\left|
	V
	-
	\frac{JM}{J+1} \cdot
	\frac{1}{1 - a + {a }/{J} }
\right|
\leq
\epsilon \kappa_3   M \, ,\tag{\ref{eq:Lem8A}$'$}
\\[.5em]
&
Q^{\Sigma}_{\mA}(V)
\geq
 \frac{aM}{1 -a + {a}/{  J}}
-
\epsilon  \kappa_4  M\,. \tag{\ref{eq:Lem8B}$'$}
\end{align}
The proof of Lemma \ref{LemmaStillEmpty}$'$ and the completion of the proof of Proposition \ref{prop1'} differ slightly, although their statements are the same as for the corresponding results in Section \ref{sec:proofs}.  The statements and proofs of the other lemmas are identical to those of their corresponding lemmas in Section \ref{sec:proofs}.  The statements and proofs of all of the lemmas are given in detail in \cite{BDW19}.

\includepdf[pages=1-last]{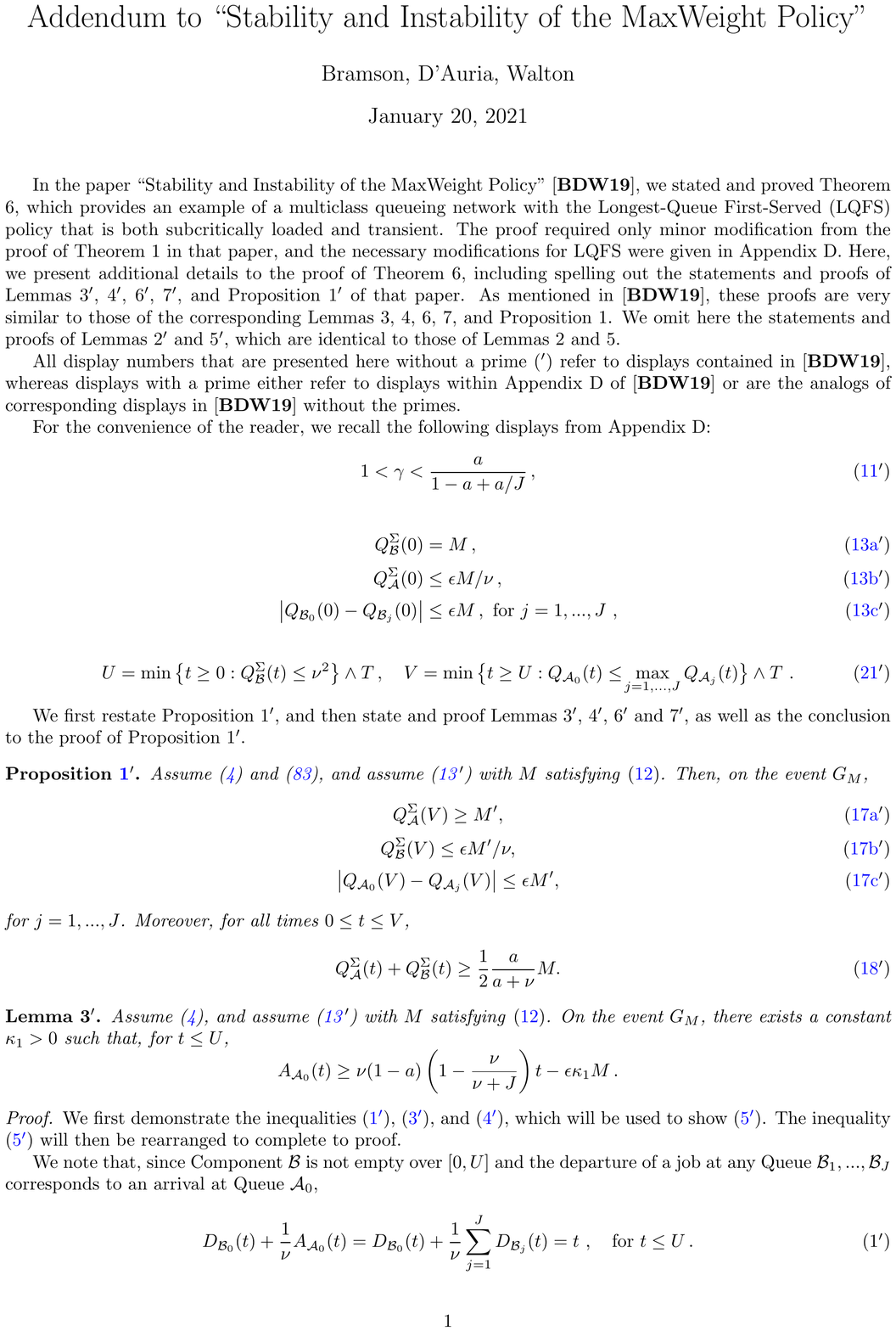}

\end{document}